\documentclass[11pt,leqno]{amsart}
\usepackage{amssymb,amsmath,amsthm,amsfonts}
\usepackage{graphicx}
\usepackage{float}
\usepackage{xcolor}

\DeclareMathOperator{\osc}{osc} \DeclareMathOperator{\Proj}{Proj}

\DeclareMathOperator{\BV}{BV}\DeclareMathOperator{\Graph}{Gr}

\newcommand{\field}[1]{\mathbb{#1}}
\newcommand{\N}{\field{N}}                      % Naturals
\newcommand{\Z}{\field{Z}}                      % Integers
\newcommand{\R}{\field{R}}                      % Reals
\newcommand{\C}{\field{C}}                      % Complex
\newcommand{\dimAreg}{\dim_{A,reg}}

\newcommand{\ovdimB}{{\overline{\dim_B}}}

\theoremstyle{plain}
\newtheorem{theorem}{Theorem}

\newtheorem{corollary}[theorem]{Corollary}
\newtheorem{lemma}[theorem]{Lemma}

\theoremstyle{definition}
\newtheorem{definition}[theorem]{Definition}
\newtheorem{example}[theorem]{Example}
\newtheorem{remark}[theorem]{Remark}
\newtheorem{remarks}[theorem]{Remarks}
\newtheorem{question}[theorem]{Question}

\newtheorem{claim}[theorem]{Claim}

\numberwithin{theorem}{section} \numberwithin{equation}{section}

\newcommand{\eps}{\epsilon}

\newcommand{\cD}{{\mathcal D}}

\newcommand{\tf}{{\tilde{f}}}

\newcommand{\bz}{{\mathbf z}}
\newcommand{\bw}{{\mathbf w}}

\newcommand{\ceil}[1]{\left\lceil #1 \right\rceil}

\newcommand{\ubdim}{\overline{\dim}_B}
\newcommand{\assdim}{\dim_A}

\def\Barint_#1{\mathchoice
	{\mathop{\vrule width 6pt height 3 pt depth -2.5pt
			\kern -8pt \intop}\nolimits_{#1}}%
	{\mathop{\vrule width 5pt height 3 pt depth -2.6pt
			\kern -6pt \intop}\nolimits_{#1}}%
	{\mathop{\vrule width 5pt height 3 pt depth -2.6pt
			\kern -6pt \intop}\nolimits_{#1}}%
	{\mathop{\vrule width 5pt height 3 pt depth -2.6pt
			\kern -6pt \intop}\nolimits_{#1}}}

\title{On the Assouad spectrum of H\"older and Sobolev graphs}
\date{\today}
\author{Efstathios-K. Chrontsios-Garitsis}
\address{Department of Mathematics \\ University of Tennessee, Knoxville \\ 1403 Circle Dr \\ Knoxville, TN 37966}
\email{echronts@utk.com, echronts@gmail.com}
\author{Jeremy T. Tyson}
\address{Department of Mathematics \\ University of Illinois at Urbana-Champaign \\ 1409 West Green Street \\ Urbana, IL 61801}
\email{tyson@illinois.edu}

\begin{document}
	\maketitle
	
	\begin{abstract}
		We provide upper bounds for the Assouad spectrum $\dim_A^\theta(\Graph(f))$ of the graph of a real-valued H\"older or Sobolev function $f$ defined on an interval $I \subset \R$. We demonstrate via examples that all of our bounds are sharp. In the setting of H\"older graphs, we further provide a geometric algorithm which takes as input the graph of an $\alpha$-H\"older continuous function satisfying a matching lower oscillation condition with exponent $\alpha$ and returns the graph of a new $\alpha$-H\"older continuous function for which the Assouad $\theta$-spectrum realizes the stated upper bound for all $\theta\in (0,1)$. Examples of functions to which this algorithm applies include the continuous nowhere differentiable functions of Weierstrass and Takagi.
	\end{abstract}
	
	\section{Introduction}
	
	A fruitful line of study relates analytic regularity conditions of a continuous function $f:I\rightarrow \R$ to geometric regularity properties of its graph $\Graph(f):=\{(t,f(t))\subset \R^2: t\in I \}$. For instance, it is natural to inquire how regularity hypotheses on $f$ influence measures for the `fractality' of $\Graph(f)$, such as the values of various metrically defined notions of dimension.
	
	Notable examples in this regard are the continuous and nowhere differentiable functions of Weierstrass and Takagi. Given $a\in (0,1)$ and $b\in (1/a,\infty)$ we define the Weierstrass functions
	$$
	W_{a,b}(t):=\sum_{n=0}^{\infty} a^n \cos (b^n t), \qquad 0 \le t \le 1,
	$$ 
	and the Takagi functions
	$$
	T_{a,b}(t)=\sum_{n=0}^{\infty}a^n D(b^n t), \qquad 0 \le t \le 1,
	$$ 
	where $D$ is the $1$-periodic sawtooth function with $D(t)=t$ for all $t\in [0,1/2]$ and $D(t)=1-t$ for all $t\in [1/2,1]$. The graphs of these functions have seen sustained attention by the fractal geometry community for many years. For instance, a longstanding program of research focuses on the determination of metric dimensions of the graphs of the Weierstrass functions. While the upper box dimension of $\Graph(W_{a,b})$ has been known since the early 20th century (see \cite{Hardy1916}), its Hausdorff dimension was not conclusively determined until 2018 by Shen \cite{Shen}.
	
	The notions of dimension which we focus on in this paper are the {\bf Assouad dimension} $\dim_A$, introduced by Assouad in \cite{AssouadOriginal} under a different name, as well as the {\bf Assouad spectrum} $\dim_A^\theta$ (and its regularization $\dimAreg^\theta$), introduced by Fraser and Yu \cite{FRASER_YU_Introd_spectrum}. These notions of dimension quantify the covering properties of a given set or space at all locations and scales, with respect to a smaller scale which may or may not be quantitatively related to the larger scale. Precise definitions are recalled in Section \ref{sec:background}. 
	
	The exact value of the Assouad dimension for the graphs of Weierstrass and Takagi functions was posed as a question by Fraser \cite[Question 17.11.1]{FraserBook} and remains open. It is worth noting that while there are (to the best of our knowledge) no known partial results for the Weierstrass function, partial results and bounds on the Assouad dimension of the graphs of some Takagi functions were proved by Yu \cite{TakagiGraphYu} and Anttila, B\'ar\'any and K\"aenm\"aki \cite{Kaenm_Takagi}.
	
	In an effort to shed new light on this subject, we study how H\"older conditions on a function $f$ influence bounds for the Assouad spectra of $\Graph(f)$. Our main contributions are summarized in the following two theorems.
	
	\begin{theorem}\label{thm:spectrum_bdd_for_Holder}
		Let $I \subset \R$ be a closed interval and let $f:I\rightarrow \R$ be an $\alpha$-H\"older continuous function. Then
		\begin{equation}\label{eq:dimAreg_bound}
			\dimAreg^\theta \Graph(f) \leq \frac{2-\alpha-\theta}{1-\theta},
		\end{equation} for all $\theta \in (0,\alpha)$. 
	\end{theorem}
	
	Note that when $\theta = \alpha$, the upper bound in \eqref{eq:dimAreg_bound} reaches $2$ and hence provides no nontrivial information. In fact, for values of $\theta \in (0,\alpha)$ the upper bound in \eqref{eq:dimAreg_bound} is sharp in the following sense:
	
	\begin{theorem}\label{th:Holder_bd_sharp}
		For every $\alpha\in (0,1)$ there exists an $\alpha$-H\"older continuous function $\tf=\tf_{\alpha}:[0,1]\rightarrow\R$ with
		\begin{equation}\label{eq:Holder_bd_sharp}
			\dim_A^\theta \Graph(\tf)=\dimAreg^\theta \Graph(\tf) = \frac{2-\alpha-\theta}{1-\theta},
		\end{equation} for all $\theta\in (0,\alpha)$.
	\end{theorem}
	
	Recall (see e.g.\ \cite[Lemma 3.4.4]{FraserBook}) that the Assouad spectrum $\dim_A^\theta(E)$ tends to the upper box dimension $\ubdim(E)$ as $\theta\searrow 0$, for any set $E \subset \R^n$. Letting $\theta \searrow 0$ in Theorems \ref{thm:spectrum_bdd_for_Holder} and \ref{th:Holder_bd_sharp}, we recover the (known) sharp upper estimate $2-\alpha$ for the box-counting dimension of an arbitrary $\alpha$-H\"older continuous function $f:I \to \R$. See, for instance, \cite[Chapter 11]{FalconerBook}.
	
	Not only is the proof of Theorem \ref{th:Holder_bd_sharp} constructive, but it provides a geometric algorithm that modifies the Assouad spectrum of the graph of an appropriate function --- such as a Weierstrass or Takagi function --- while leaving the upper box dimension intact. This feature is worth emphasizing, since the algorithm involves countably many modifications to the graph and the upper box dimension is not countably stable in general. We believe that this method is of independent interest with potential applications to the study of other fractal sets. The key property which the algorithm seeks to ensure for the modified graph is that it lies for a uniformly sufficient amount of `time' within a predetermined sequence of squares. Such a property can be exploited to numerically estimate the extent to which the graph of a given function, e.g., a Weierstrass function, deviates from satisfying \eqref{eq:Holder_bd_sharp}, or from having Assouad dimension equal to $2$.
	
	\medskip
	
	We also consider functions in the Sobolev class $W^{1,p}$ for $1 \le p \le \infty$. For a closed interval $I$, recall that $f \in W^{1,p}(I)$ if $f'$ exists weakly as a function and $f,f' \in L^p(I)$. Since we are only interested in continuous functions $f$ we drop the integrability condition on $f$. We have
	$$
	W^{1,\infty} \subset W^{1,q} \subset W^{1,p} \subset W^{1,1} \subset \BV, \qquad p<q,
	$$
	where $\BV(I)$ denotes the space of functions of {\it bounded variation} on $I$. The space of continuous representatives of functions in $W^{1,\infty}(I)$ coincides with the space of bounded Lipschitz functions on $I$. An elementary estimate using the Fundamental Theorem of Calculus and H\"older's inequality shows that $W^{1,p}$ functions are $\alpha$-H\"older continuous with $\alpha = 1-1/p$. An application of Theorem \ref{thm:spectrum_bdd_for_Holder} then gives
	\begin{equation}\label{eq:adtGf1}
		\dimAreg^\theta\Graph(f) \le 1 + \frac{1}{(1-\theta)p}.
	\end{equation}
	On the other hand, the upper box dimension $\ubdim(\Graph(f))$ is equal to one for any BV function $f:I \to \R$; see Theorem \ref{th:ubd-bv}. By a standard estimate for the Assouad spectrum \cite[Lemma 3.4.4]{FraserBook}, we always have the bound
	\begin{equation}\label{eq:adtGf2}
		\dimAreg^\theta\Graph(f) \le \frac{\ubdim(\Graph(f))}{1-\theta} = 1 + \frac{\theta}{1-\theta}
	\end{equation}
	for any BV function $f$ and any $0<\theta<1$.
	
	Our next result provides an upper bound for $\dimAreg^\theta(\Graph(f))$ in the case when $f \in W^{1,p}(I)$, which improves on both \eqref{eq:adtGf1} and \eqref{eq:adtGf2}.
	
	\begin{theorem}\label{th:ass-spectrum-sobolev}
		Let $f \in W^{1,p}(I)$, $1\le p\le\infty$, be a continuous, real-valued function defined on an interval $I \subset \R$. Then
		\begin{equation}\label{eq:ass-spectrum-sobolev}
			\dimAreg^\theta\Graph(f) \le 1 + \frac{\theta}{(1-\theta)p}.
		\end{equation}
		for all $0<\theta<\tfrac{p}{p+1}$.
	\end{theorem}
	
	As before, observe that the upper bound in \eqref{eq:ass-spectrum-sobolev} is only of interest when $\theta < \tfrac{p}{p+1}$; once $\theta = \tfrac{p}{p+1}$ the upper bound reaches $2 = \dim \R^2$ and the conclusion becomes trivial. For $\theta<\tfrac{p}{p+1}$ the conclusion in Theorem \ref{th:ass-spectrum-sobolev} is sharp.
	
	\begin{theorem}\label{th:ass-spectrum-sobolev-2}
		For each $1<p<\infty$, there exists $f \in W^{1,p}([0,1])$ so that 
		\begin{equation}\label{eq:ass-spectrum-sobolev-2}
			\dimAreg^\theta(\Graph(f)) = 1 + \frac{\theta}{(1-\theta)p}
		\end{equation}
		for all $0<\theta<p/(p+1)$.
	\end{theorem}
	
	This paper is organized as follows. Section \ref{sec:background} reviews basic notation and terminology, including definitions for the Assouad dimension $\dim_A(E)$ and Assouad spectrum $\dim_A^\theta(E)$ of a set $E$. 
	
	Section \ref{sec:Holder} contains our results on the Assouad spectra of H\"older graphs. In particular, in subsection \ref{subsec:holder-sharpness} we present an algorithm for the construction of $\alpha$-H\"older continuous functions $\tf:I \to \R$ for which the Assouad spectrum, $\dim_A^\theta(\Graph(\tf))$, realizes the upper bound in \eqref{eq:dimAreg_bound} for all $\theta$. Fixing $\alpha\in (0,1)$ and $\theta_0 \in (0,\alpha)$, the algorithm takes as input an $\alpha$-H\"older continuous function $f$ which satisfies a matching lower oscillation condition with the same exponent $\alpha$. The graph of $f$ is modified by a sequence of transformations, each of which reflects a selected subgraph across the boundary of a selected square $Q$. The choice of where to implement these reflections inside of $Q$ depends in a subtle way on the parameter $\theta_0$. The reflections in question are chosen to force the new graph to lie entirely inside the right or left half of squares in the designated sequence while maintaining the H\"older regularity. The output of this algorithm is a new function $\tf$ whose graph is localized within a sequence of squares, and which consequently (by construction) attains the desired value of $\dim_A^{\theta_0}(\Graph(\tf))$. However, an elementary observation regarding the behavior of the covering number for $\Graph(\tf) \cap Q$ with respect to different choices of the smaller scale allows us to promote this conclusion from the single choice $\theta = \theta_0$ to the full range of Assouad spectra $\dim_A^{\theta}(\Graph(\tf))$ for {\bf all} $\theta\in (0,\alpha)$.
	We explore the question of lower bounds for the Assouad spectrum further in subsection \ref{subsec:lower-bounds}, where we identify related (but weaker) lower oscillation and graph localization conditions which ensure the validity of other lower bounds on the Assouad spectrum.
	
	In section \ref{sec:Sobolev} we turn to the study of graphs of Sobolev functions. After some preliminary comments about box-counting dimensions of Sobolev and BV graphs, we prove Theorems \ref{th:ass-spectrum-sobolev} and \ref{th:ass-spectrum-sobolev-2}. The construction of a Sobolev graph which realizes the upper bound in \eqref{eq:ass-spectrum-sobolev} is considerably easier than the previous algorithmic construction of a H\"older extremal function. A suitably chosen piecewise linear graph, oscillating infinitely often between the graphs of $y=x$ and $y=-x$ and converging to the origin, realizes the desired dimensional value.
	
	\smallskip
	
	\paragraph{\bf Acknowledgments.} 
	
	We gratefully acknowledge valuable insights from Jonathan Fraser on the subject of Assouad dimensions of graphs, and especially for bringing to our attention the recent preprint \cite{Kaenm_Takagi} on the Assouad dimensions of Takagi graphs. We also thank Roope Anttila for pointing out that an earlier, weaker version of Question \ref{Que: weak lower bound T and W} was in fact known. JTT acknowledges support from the Simons Foundation under grant \#852888, `Geometric mapping theory and geometric measure theory in sub-Riemannian and metric spaces'. In addition, this material is based upon work supported by and while JTT was serving as a Program Director at the U.S. National Science Foundation. Any opinion, findings, and conclusions or recommendations expressed in this material are those of the authors and do not necessarily reflect the views of the National Science Foundation. Lastly, we wish to thank the anonymous referee for reading our manuscript carefully, and for providing insightful feedback that has improved the exposition.
	
	\section{Background}\label{sec:background}
	We first establish notation which will be in use throughout the paper. For $x\in \R$, denote by $\lceil x \rceil$ the ceiling function of $x$, i.e. the smallest integer larger than or equal to $x$. Similarly, we denote by $\lfloor x\rfloor$ the floor function of $x$, i.e. the largest integer smaller than or equal to $x$. 
	
	Points in the plane $\R^2$ will be denoted by $\bz,\bw,\ldots$. We denote by $Q(\bz,r)$ the closed axes-oriented square centered at a point $\bz\in \R^2$ of side length $2r$. For a nonempty set $A\subset \R^2$ we denote by $\Proj_x(A)$ and $\Proj_y(A)$ the projection of $A$ onto the $x$-axis and the $y$-axis, respectively. The complement of a set $E \subset \R^2$ is denoted $E^\complement := \R^2 \setminus E$.
	
	For functions $A(R,r)$ and $B(R,r)$ with $0<r\leq R<1$, we write $A(R,r)\lesssim B(R,r)$ if there is a constant $C>0$ that does not depend on $R$ or $r$ such that $A(R,r)\leq C B(R,r)$ for all $0<r\leq R<1$.
	
	For a bounded interval $I \subset \R$ we denote by $|I|$ its length. Moreover, for a function $f:I\rightarrow \R$, we denote the oscillation of $f$ over $J$ by $\osc(f,J)=\sup\{ |f(t)-f(s)|: \, t,\, s \in J \}$ for any interval $J\subset I$. We also write $\Graph(f;J) := \{ (t,f(t)): t\in J \}$ for any interval $J\subset I$.
	
	Recall that the Assouad dimension of a set $E\subset \R^2$ is defined as
	\begin{equation}\label{def:Assouad_dim}
		\dim_A(E) := \inf \left\{\alpha>0 \,:\, {\exists\,C>0\mbox{ s.t. } N(D(\bz,R) \cap E,r) \le C (R/r)^{\alpha} \atop \mbox{ for all $0<r\leq R$ and all $\bz \in E$}} \right\},
	\end{equation} 
	where $N(F,r)$ denotes the least number of sets of diameter at most $r$ needed to cover $F$. Assouad dimension was introduced (under a different name and with a different definition) by Assouad in \cite{AssouadOriginal}; the formulation here is taken from \cite{FraserBook}. We also recall the notion of Assouad spectrum, introduced by Fraser and Yu \cite{FRASER_YU_Introd_spectrum}. It is a one-parameter family of metrically defined dimensions which interpolates between the upper box-counting dimension and the (quasi-)Assouad dimension. Specifically, the Assouad spectrum of a set $E \subset \R^2$ is a collection of values
	$$
	\{\dim_A^\theta(E):0<\theta<1\},
	$$
	where
	\begin{equation}\label{def:Assouad-spectrum}
		\dim_{A}^\theta(E) := \inf \left\{\alpha>0 \,:\, {\exists\,C>0\mbox{ s.t. } N(D(\bz,R) \cap E,r) \le C (R/r)^{\alpha} \atop \mbox{ for all $0<r= R^{1/\theta}< R< 1$ and all $\bz \in E$}} \right\}.
	\end{equation} 
	While one could replace $R$ by $r^\theta$ in  \eqref{def:Assouad-spectrum} for simplicity, which we actually implement for arguments in later sections, we prefer to state the definition of the Assouad spectrum in this form to indicate the relation to the Assouad dimension.
	
	\begin{remark}
		Note that the covering number $N(D(\bz,R) \cap E,r)$ can be replaced by the number of closed (or open) sub-squares of $Q(\bz,R)$ of side length $r$ needed to cover $Q(\bz,R) \cap E$ (or the interior of $Q(\bz,R) \cap E$) without affecting the values of the dimensions $\dim_A(E)$ and $\dim_A^\theta(E)$ in \eqref{def:Assouad_dim} and \eqref{def:Assouad-spectrum}, respectively. See, e.g., \cite[pp.\ 3 and 11]{FraserBook} and \cite[Proposition 2.5]{QCourPAPER}. In later sections, we will alternate between equivalent definitions depending on what is most convenient for each argument.
	\end{remark} 
	
	The map $\theta \mapsto \dim_A^\theta(E)$ is continuous (even locally Lipschitz) when $0<\theta<1$, and
	$$
	\dim_A^\theta(E) \to \ovdimB(E) \quad \mbox{as $\theta \to 0$}, \qquad \dim_A^\theta(E) \to \dim_{qA}(E) \quad \mbox{as $\theta \to 1$} \, .
	$$
	Here $\ovdimB(E)$ denotes the upper box-counting dimension of $E$, while $\dim_{qA}(E)$ denotes the {\it quasi-Assouad dimension} of $E$; the latter is a variant of Assouad dimension introduced by L\"u and Xi \cite{Lu_Xi_QuasiAssouad}. We always have $\dim_{qA}(E) \le \dim_A(E)$, and equality holds in many situations (see \cite[Section 3.3]{FraserBook} for details).
	
	A slightly modified version of the Assouad spectrum where the relationship $R=r^\theta$ between the two scales is relaxed to an inequality $R\ge r^\theta$ leads to the notion of {\it upper Assouad spectrum}, denoted $\overline{\dim_A^\theta}(E)$ in the literature: see \cite[Section 3.3.2]{FraserBook} for more information. The key relationship between the two values (see \cite[Theorem 3.3.6]{FraserBook}) is that
	\begin{equation}\label{eq:regularization}
		\overline{\dim_A^\theta}(E) = \sup_{0<\theta'<\theta} \dim_A^{\theta'}(E).
	\end{equation}
	The authors have proposed in \cite{QCourPAPER} the term {\it regularized Assouad spectrum} in lieu of upper Assouad spectrum, and the notation $\dim_{A,reg}^\theta(E)$ in place of $\overline{\dim_A^\theta}(E)$; this terminology and notation will be used throughout this paper.
	
	\section{Dimensions of H\"older graphs}\label{sec:Holder}
	
	In this section we prove Theorems \ref{thm:spectrum_bdd_for_Holder} and \ref{th:Holder_bd_sharp}, which provide sharp upper bounds for the Assouad spectra of H\"older graphs.
	
	\subsection{Proof of Theorem \ref{thm:spectrum_bdd_for_Holder}} 
	
	Set $\beta=\frac{2-\alpha-\theta}{1-\theta}$ and let $\bz \in \Graph(f)$ with $\bz=(t, f(t))$ and $R\in (0,1)$, $\rho\in (0,R^{1/\theta}]$.
	
	We will prove that there is a constant $C$, depending only on $\beta$, such that the minimum number of sets of diameter at most $\rho$ needed to cover $D(\bz,R)\cap \Graph(f)$ does not exceed $C(R/\rho)^\beta$, i.e.
	$$
	N:=N(D(\bz,R)\cap \Graph(f),\rho)\leq C(R/\rho)^\beta.
	$$ 
	This estimate easily implies the desired conclusion. Set $r= \rho/\sqrt{2}$. We will use essentially disjoint squares of side-length $r$ to cover $D(\bz,R)\cap \Graph(f)$ and count how many such squares are needed.
	
	The projection of $D(\bz,R)\cap \Graph(f)$ onto the $x$-axis is included in the interval $I=[t-R,t+R]$. We need at most $M:=2 \lceil R/r \rceil$ intervals of the form $I_i=[t-R+ir,t-R+(i+1)r]$, $i=0, 1, \dots, M-1$ to cover $I$. Split all columns $J_i := I_i \times \R$ into squares of side-length $r$. In a given column $J_i$ there can be at most $\osc(f,I_i)/r+2$ such squares intersecting $\Graph(f)$. Using all the squares in columns $J_i$, $i=0, 1, \dots M$, to cover $D(\bz,R)\cap\Graph(f)$, and noting that each of these squares has diameter $\sqrt{2}r=\rho$, we deduce that
	$$
	N\leq \sum_{i=0}^{M-1}(\osc(f,I_i)/r+2).
	$$ 
	By the $\alpha$-H\"older continuity of $f$ we know that $\osc(f,I_i)\lesssim r^\alpha$ which, along with $M\leq 4R/r$, implies that
	\begin{equation}\label{eq:N_less_than_R/r^exp}
		N\lesssim \sum_{i=0}^{M-1}(r^{\alpha-1}+2)\leq \frac{4R}{r^{2-\alpha}}+\frac{8R}{r} \lesssim \left(\frac{R}{r}\right)^{2-\alpha} R^{\alpha-1}.
	\end{equation} 
	Note that the estimate $R/r\leq R/r^{2-\alpha}$ follows from the restrictions $r<1$ and $1-\alpha>0$.
	
	Recalling that the choice of $\rho=\sqrt{2}r\leq R^{1/\theta}$ was arbitrary, we conclude that
	\begin{equation}\label{eq:Rtor}
		R^{-1/\theta}\lesssim \frac{1}{r}.
	\end{equation}
	Multiplying both sides of \eqref{eq:Rtor} by $R$ and raising to the (positive) power of $1-\alpha$ yields
	$$
	R^{(1-\alpha)(1-1/\theta)}\lesssim \left(\frac{R}{r}\right)^{1-\alpha}.
	$$ 
	Since $R<1$, $\theta\in (0,1)$, and $1-1/\theta=(\theta-1)/\theta<0$, the above implies that
	$$
	R^{\alpha-1}\lesssim \left(\frac{R}{r}\right)^\frac{(1-\alpha)\theta}{1-\theta}
	$$ 
	Combining this with \eqref{eq:N_less_than_R/r^exp} yields
	$$
	N\leq \left(\frac{R}{r}\right)^{\frac{(1-\alpha)\theta}{1-\theta}+2-\alpha} = \left( \frac{R}{r} \right)^\beta,
	$$ 
	since
	$$
	\frac{(1-\alpha)\theta}{1-\theta}+2-\alpha= \frac{(1-\alpha)\theta+1-\theta+(1-\alpha)(1-\theta)}{1-\theta}= \beta.
	$$ 
	Hence $N\lesssim \left(\frac{R}{r}\right)^\beta \lesssim\left(\frac{R}{\rho}\right)^\beta$ as needed. This concludes the proof of Theorem \ref{thm:spectrum_bdd_for_Holder}.
	
	\subsection{Proof of Theorem \ref{th:Holder_bd_sharp}}\label{subsec:holder-sharpness}
	
	Fix $\alpha\in (0,1)$ and fix an arbitrary function $f:[0,1]\rightarrow\R$ for which there are $C>0$, $c>0$ and $r_0>0$ such that
	\begin{equation}\label{eq:f_upper_Holder}
		|f(t)-f(s)|\leq C |t-s|^\alpha
	\end{equation}
	for all $t, s \in [0,1]$ and
	\begin{equation}\label{eq:f_lower_Holder}
		\osc(f,I)\geq c |I|^\alpha
	\end{equation} 
	for all intervals $I\subset [0,1]$ with $|I|\leq r_0$. An example of such a function is the Takagi function $T_{a,2}$ where $a:=2^{-\alpha}$; the $\alpha$-H\"older continuity is proved in \cite[Proposition 2.3]{Baranski}, while \cite[Theorem 2.4]{Baranski} asserts the existence of $c>0$ such that for all $t, s\in (0,1)$ we have $|f(t)-f(s)|\geq c |t-s|^\alpha$. The latter condition is a lower H\"older continuity estimate that implies \eqref{eq:f_lower_Holder}. Another possible candidate for $f$ is the Weierstrass function $W_{b^{-\alpha},b}$ for sufficiently large $b>1$; see \cite[Theorem 1.31 and Theorem 1.32]{Hardy1916} and \cite[Example 11.3]{FalconerBook} for details.
	
	Now fix $\theta_0\in (0,\alpha)$. We will modify the graph of $f$ on a sequence of squares $\{Q_k\}_{k\in \N}$ with decreasing side lengths $r_k^{\theta_0}\searrow 0$ to obtain the graph of a new function $\tf$. This modification aims to maximize the least number of sub-squares of $Q_k$  of side length $r_k$ needed to cover the resulting graph lying in $Q_k$, for all $k\in \N$. The output is a new function $\tf$ which again satisfies \eqref{eq:f_upper_Holder} and also obtains the desired equality \eqref{eq:Holder_bd_sharp} for the Assouad spectrum with parameter $\theta_0$. Then, for an arbitrary $\theta\in (0,\alpha)$, a simple choice of scales $\tilde{r}_k=r_k^{\theta_0/\theta}$ shows that the same $\tf$ satisfies the desired equality for $\dimAreg^\theta \Graph(f)$ as well without further modifications.
	
	The lower oscillation estimate \eqref{eq:f_lower_Holder} implies that $f$ is nowhere differentiable, consequently, its local extrema are dense in $[0,1]$. Let $m_0$ and $m_1$ be local maxima of $f$ in the intervals $(1/2-\min\{ r_0, 10^{-1} \}/2, 1/2)$ and $(1/2, 1/2+\min\{ r_0, 10^{-1} \}/2)$ respectively. Define
	$$
	\delta(m):= \sup \{ \rho>0 : f(x)< f(m) \,\,\,\,\, \forall x\in (m-\rho, m+\rho) \},
	$$ for each maximum $m\in (0,1)$ of $f$. Note that $\min\{\delta(m_0), \delta(m_1)\}\leq \min\{ r_0, 1/10\}$ by the choice of $m_0$ and $m_1$.
	
	Without loss of generality assume that $f(m_1)<f(m_0)$, so that $\delta(m_1)\leq \min\{ r_0, 1/10\}$ and $\Graph(f)$ intersects the left side of the square $Q_1:= Q(\bz_1, r_1^{\theta_0})$ at its midpoint $\bz_1':= (m_1-r_1^{\theta_0}, f(m_1))$, where $\bz_1:=(m_1, f(m_1))$ and $r_1:=\delta(m_1)^{1/\theta_0}$. The proof follows similarly if $f(m_1)>f(m_0)$. By the choice of $r_1$, the graph of $f$ stays below the horizontal line $y=f(m_1)$ within the rectangle
	$$
	R_1:=[m_1-r_1^{\theta_0}, m_1]\times [f(m_1)-r_1^{\theta_0}, f(m_1)+r_1^{\theta_0}] \subset Q_1.
	$$
	
	\begin{figure}[H]
		\centering
		\includegraphics[width=0.5\textwidth]{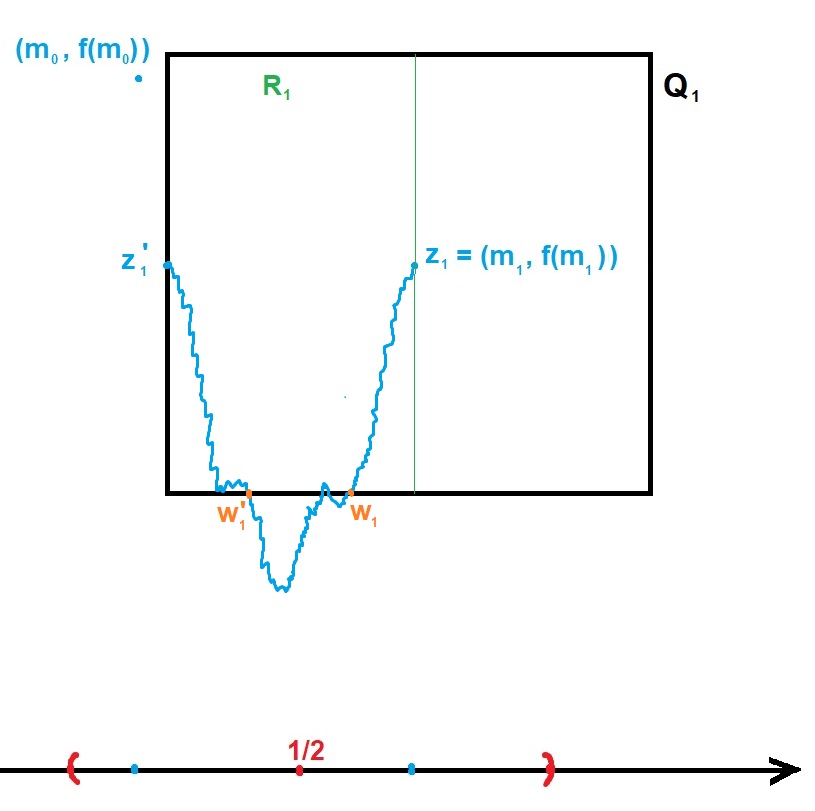}
		\caption{Choice of $\bz_1$ with $x$-coordinate lying in the red interval centered at $1/2$ of length $\min\{ r_0, 10^{-1} \}$.}
		\label{fig:Choosing point and radius 1}
	\end{figure}
	
	We shall modify the graph of $f$ in the strip $(m_1-r_1^{\theta_0},m_1)\times \R$ and this modification will be the graph of a new function $f_1$, with the property that $\Graph(f_1;[m_1-r_1^{\theta_0},m_1])$ lies in $R_1$. The square $Q_1$ is the first square in the sequence alluded to above. If $\Proj_x(\Graph(f)\cap R_1)=[m_1-r_1^{\theta_0},m_1]$, i.e., if the restricted graph $\Graph(f;[m_1-r_1^{\theta_0},m_1])$ already lies in $R_1$, then we set $f_1:=f$ and we proceed to the choice of $Q_2$. 
	
	Suppose $\Proj_x(\Graph(f)\cap R_1)\neq [m_1-r_1^{\theta_0},m_1]$. The idea is to reflect $\Graph(f)$ with respect to the lower and upper horizontal sides of $R_1$ in an alternating fashion until all of the graph between $x=m_1-r_1^{\theta_0}$ and $x=m_1$ has been ``folded'' inside $R_1$. More specifically, let $x_1'\in(m_1-r_1^{\theta_0},m_1)$ be the smallest value for which $f(x_1')=f(m_1)-r_1^{\theta_0}$ and there is $\epsilon_1>0$ such that
	$$
	f(x)\geq f(m_1)-r_1^{\theta_0}
	$$ 
	for all $x\in (m_1-r_1^{\theta_0}, x_1')$, and
	$$
	f(x)<f(m_1)-r_1^{\theta_0}
	$$ 
	for all $x\in (x_1', x_1'+\eps_1)$. In other words, the point $\bw_1':=(x_1', f(x_1'))$ is the first point from $\bz_1'=(m_1-r_1^{\theta_0}, f(m_1))$ towards $\bz_1=(m_1, f(m_1))$ at which the graph of $f$ intersects the lower side of $R_1$ and stays outside $R_1$ afterwards within some small strip $(x_1', x_1'+\eps_1)\times \R$. Similarly, define $\bw_1:= (x_1, f(x_1))$ to be the last point from $\bz_1'$ towards $\bz_1$ where $\Graph(f)$ intersects the lower side of $R_1$ and stays inside $R_1$ afterwards within $(x_1, m_1)$. See Figure \ref{fig:Choosing point and radius 1}.
	
	We reflect $\Graph(f;(x_1',x_1))\cap R_1^\complement$ with respect to the line $y=f(m_1)-r_1^{\theta_0}$. Note that by \eqref{eq:f_lower_Holder}, if $r_1^{\theta_0}$ is small enough, it is very likely that the reflected part of the graph will now intersect and exit the upper side of $R_1$, since it could be the case that $c(r_1^{\theta_0})^\alpha-r_1^{\theta_0}>2r_1^{\theta_0}$. In that case, following the same argument but for the upper side of $R_1$, we again reflect the part of $\Graph(f)$ that escapes $R_1$ from above, this time with respect to the horizontal line $y=f(m_1)+r_1^{\theta_0}$. Again, it may be the case that the newly reflected part now escapes $R_1$ from the lower side of $R_1$, in which case we perform yet another reflection with respect to $y=f(m_1)-r_1^{\theta_0}$. If we denote by $M_1:=\lceil Cr_1^{(\alpha-1){\theta_0}}-2 \rceil$ the smallest positive integer such that $C(r_1^{\theta_0})^\alpha-M_1 r_1^{\theta_0}\leq 2 r_1^{\theta_0}$, then we are guaranteed that after at most $M_1+1$ reflections, the entire graph of $f$ from $\bz_1'$ to $\bz_1$ must lie inside $R_1$ (due to the H\"older continuity of $f$). The preceding process gives rise to the graph of a new continuous function which we denote $f_1$. Note that $f_1$ satisfies
	$$
	\Proj_x(\Graph(f_1)\cap R_1)= [m_1-r_1^{\theta_0},m_1],
	$$ 
	and
	$$
	f_1|_{R_1^\complement} = f.
	$$
	
	We now pick a second square $Q_2$ and modify $f_1$ within $Q_2\cap R_1^\complement$. By the choice of $r_1^{\theta_0}=\delta(m_1)$, $f_1$ has a global maximum at $m_1$ within the interval $[m_1, m_1+10^{-1}r_1^{\theta_0}/2]$. Consequently, if $m_2$ is a local maximum of $f_1$ in $(m_1, m_1 +10^{-1}r_1^{\theta_0}/2)$, then $r_2^{\theta_0}:= \delta(m_2)<10^{-1} r_1^{\theta_0}$. Set $Q_2:= Q(\bz_2,r_2^{\theta_0})$. Since $m_1<m_2$ and $f(m_1)>f(m_2)$, $\Graph(f_1)$ intersects the left side of the rectangle
	$$
	R_2:= [m_2-r_2^{\theta_0},m_2]\times [f(m_2)-r_2^{\theta_0}, f(m_2)+r_2^{\theta_0}]
	$$ 
	at its mid point $\bz_2'=(m_2-r_2^{\theta_0}, f(m_2))$. In a similar fashion to the modifications which we performed on $f$ to construct $f_1$, we perform at most $M_2+1 := \lceil Cr_2^{(\alpha-1){\theta_0}}-1 \rceil$ reflections with respect to the lower and upper sides of $R_2$ to construct the graph of a new continuous function $f_2$. We conclude that
	$$
	\Proj_x(\Graph(f_2)\cap R_2)=[m_2-r_2^{\theta_0}, m_2],
	$$
	$$
	f_2|_{R_1} = f_1,
	$$
	and
	$$
	f_2|_{(R_1\cup R_2)^\complement} = f_1 = f.
	$$
	Note that $R_1\cap R_2= \emptyset$, despite the fact that $Q_2\subset Q_1$. This feature of the construction ensures that modifications taking place within $R_2$ do not affect any previous changes made within $R_1$.
	
	Assume that we have inductively constructed $Q_k=Q(\bz_k,r_k^{\theta_0})$ and $f_k$, where $k\in \N$, $k\geq 2$, $\bz_k=(m_k, f(m_k))$, $r_k^{\theta_0}= \delta(m_k)<10^{-k+1}r_{k-1}^{\theta_0}$ and
	$$
	\Proj_x(\Graph(f_k)\cap R_k)=[m_k-r_k^{\theta_0}, m_k],
	$$
	$$
	f_k|_{R_j} = f_j
	$$ 
	for all positive integers $j<k$, and
	$$
	\biggl. f_k\biggr|_{\left( \bigcup_{j=1}^k R_j \right)^\complement} = f.
	$$
	Pick a local maximum $m_{k+1}$ of $f_k$ in $(m_k, m_k+10^{-k}r_k^{\theta_0}/2)$, set $r_{k+1}^{\theta_0}:= \delta(m_{k+1})< 10^{-k}r_k^{\theta_0}$, and perform a similar suite of reflections on the rectangle
	$$
	R_{k+1}:= [m_{k+1}-r_{k+1}^{\theta_0}, m_{k+1}]\times [f(m_{k+1})-r_{k+1}^{\theta_0}, f(m_{k+1})+r_{k+1}^{\theta_0}]
	$$ 
	to construct the graph of a new continuous function $f_{k+1}$. Note that $R_{k+1}$ lies inside the square $Q_{k+1}:= Q(\bz_{k+1},r_{k+1}^{\theta_0})$, where $\bz_{k+1}=(m_{k+1}, f(m_{k+1}))$. See Figure \ref{fig:Choosing Q_k+1}.
	
	\begin{figure}[H]
		\centering
		\includegraphics[width=0.75\textwidth]{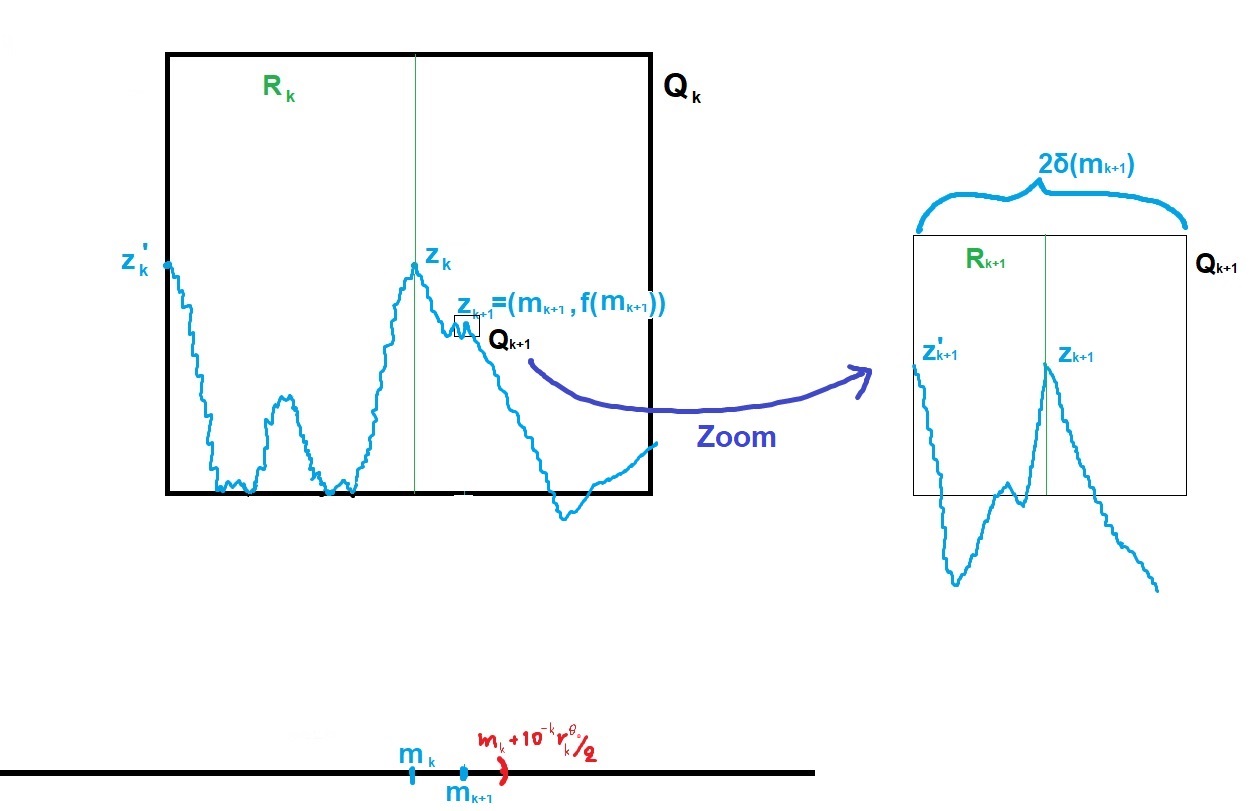}
		\caption{Choice of $Q_{k+1}$ inductively.}
		\label{fig:Choosing Q_k+1}
	\end{figure}
	
	We will show that the uniform limit $\tf$ of the sequence $\{ f_k \}_{k\in \N}$ satisfies the desired conditions with respect to the sequence of squares $\{ Q_k\}_{k\in \N}$, and that these conditions ensure that $\dim_A^{\theta_0} \Graph(\tf) \geq \frac{2-\alpha-{\theta_0}}{1-{\theta_0}}$. Note that by the choice of $r_k^{\theta_0}$, all squares in the sequence lie inside $Q_1$, and stay away from the end points of the graph due to the fact that $r_{k+1}^{\theta_0}<10^{-k}r_k^{\theta_0}<10^{-k}r_1^{\theta_0}$.
	
	We first show that $\tf$ is $\alpha$-H\"older and, more specifically, that
	\begin{equation}\label{eq:tf_upper_Holder}
		|\tf(t)-\tf(s)|\leq 3C |t-s|^\alpha \qquad \mbox{for all $t,s \in [0,1]$,}
	\end{equation}
	where $C$ is the constant from \eqref{eq:f_upper_Holder}.
	
	Let $t,s\in [0,1]$ with $t<s$. If $\tf(t)=f(t)$ and $\tf(s)=f(s)$, then \eqref{eq:tf_upper_Holder} follows from \eqref{eq:f_upper_Holder} trivially. Suppose $\tf(t)\ne f(t)$ and $\tf(s)=f(s)$. Then $t\in \Proj_x(R_k)$ and $f(t)\notin\Proj_y(R_k)$ for some $k\in \N$. Moreover, this integer $k$ is unique due to the disjointness of the rectangles $R_k$. If $(s,f(s))\in R_k$, then each of the reflections which we perform on $\Graph(f)$ moves the point $(t,f(t))\notin R_k$ closer to $(s,f(s))$ when measured along the $y$-axis, so
	$$
	|\tf(t)-\tf(s)|\leq |f(t)-f(s)|\leq 3C |t-s|^\alpha
	$$ 
	by \eqref{eq:f_upper_Holder}. Now suppose that $(s, f(s))\notin R_k$. Then
	$$
	|\tf(t)-\tf(s)|\leq |\tf(t)-\tf(m_k)|+|\tf(m_k)-\tf(s)|=|\tf(t)-f(m_k)|+|f(m_k)-f(s)|.
	$$
	However, the point $(t, f(t))$ gets closer to $\bz_k=(m_k, f(m_k))$ along the $y$-axis with every reflection, implying that $|\tf(t)-f(m_k)|\leq |f(t)-f(m_k)|$. Consequently, the preceding inequality in concert with \eqref{eq:f_upper_Holder} implies that
	$$
	|\tf(t)-\tf(s)|\leq |f(t)-f(m_k)|+|f(m_k)-f(s)|\leq C(|t-m_k|^\alpha+|m_k-s|^\alpha)\leq 2C |t-s|^\alpha,
	$$ 
	since $|t-m_k|, |m_k-s| \leq |t-s|$. Then \eqref{eq:tf_upper_Holder} trivially follows in this case as well. See Figure \ref{fig:tf_is_Holder_tf(s)=f(s)}. The proof is similar if $\tf(t)=f(t)$ and $\tf(s)\neq f(s)$, with the only difference being that the point $\bz_k'$ is used instead of $\bz_k$.
	
	\begin{figure}[H]
		\centering
		\includegraphics[width=0.45\textwidth]{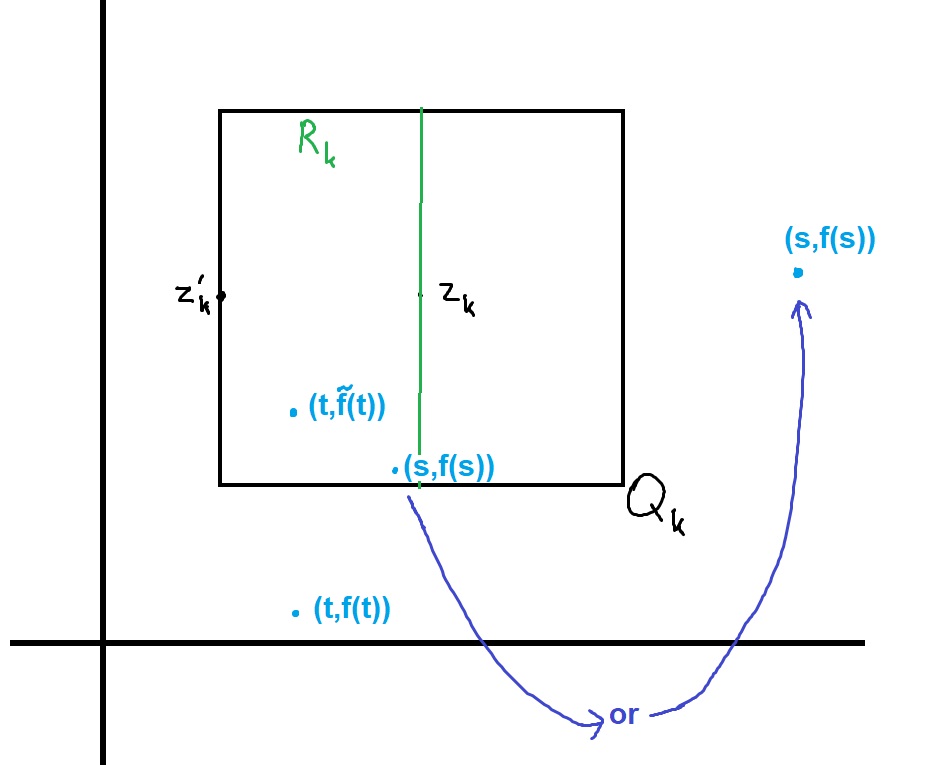}
		\caption{Showing H\"older condition in case $\tf(t)\neq f(t)$ and $\tf(s)=f(s)$, with the two possibilities for $(s,f(s))$ (lying in $R_k$ or not).}
		\label{fig:tf_is_Holder_tf(s)=f(s)}
	\end{figure}
	
	Now suppose that both $\tf(t)\neq f(t)$ and $\tf(s)\neq f(s)$. This implies that there are $k,\ell \in \N$ with $k\leq \ell$ such that $(t, \tf(t))\in R_k$ and $(s, \tf(s))\in R_\ell$. Suppose $k=\ell$. If $\Graph(f;[t,s])$ never crosses the upper or lower side of $R_k$ during any of the at most $M_k+1=\lceil Cr_k^{(\alpha-1){\theta_0}}-1 \rceil$ reflections performed to construct $f_k$, then $\Graph(\tf;[t,s])$ is isometric to $\Graph(f;[t,s])$. This implies that $|\tf(t)-\tf(s)|=|f(t)-f(s)|$ and \eqref{eq:tf_upper_Holder} trivially follows from \eqref{eq:f_upper_Holder}. For $i\leq M_k+1$, denote by $f_{i,k}$ the function resulting after $i$ reflections of $\Graph(f)$ during the construction of $f_k$ and suppose that some $i_0\leq M_k+1$ is the smallest integer for which $\Graph(f_{i_0,k};[t,s])$ intersects the upper or lower side of $R_k$. While $ |f_{i_0,k}(t)- f_{i_0,k}(s)|= |f(t)-f(s)|$ due to isometry of the two pieces of graphs, any subsequent reflection would only bring the point of $\Graph(f_{i_0,k})$ lying outside $R_k$ closer to the one lying within $R_k$. Hence, $|\tf(t)-\tf(s)|\leq |f(t)-f(s)|\leq 3C |t-s|^\alpha$ as desired. See Figure \ref{fig:tf_is_Holder_tf(s)!=f(s)}.
	
	\begin{figure}[H]
		\centering
		\includegraphics[width=0.45\textwidth]{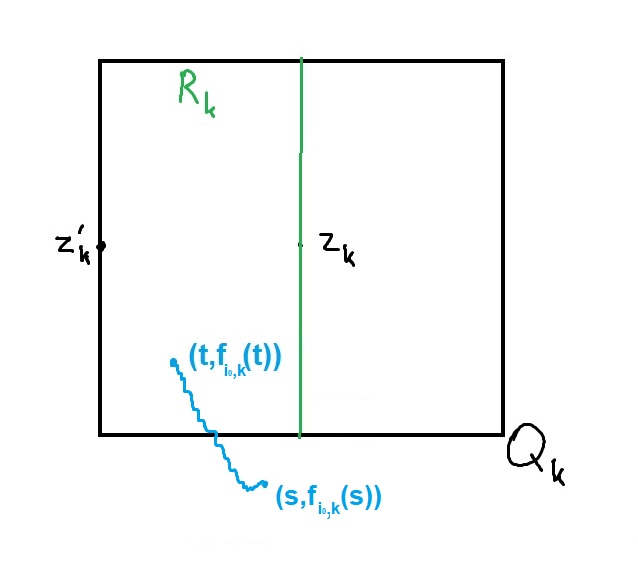}
		\caption{Showing H\"older condition in case $\tf(t)\neq f(t)$ and $\tf(s)\neq f(s)$, with $k=\ell$, and $\Graph(f_{i_0,k};[t,s])$ intersects the lower side of $R_k$ for some minimal $i_0$.}
		\label{fig:tf_is_Holder_tf(s)!=f(s)}
	\end{figure}
	
	Now suppose instead that $k<\ell$, and set $m_\ell'=m_\ell-r_\ell^{\theta_0}$. Then
	$$
	|\tf(t)-\tf(s)|\leq |\tf(t)-\tf(m_k)|+|\tf(m_k)-\tf(m_\ell')|+|\tf(m_\ell')-\tf(s)|.
	$$
	Since $\tf(m_k)=f(m_k)$ and $\tf(m_\ell')=f(m_\ell')$, similarly to the case where $\tf(t)\neq f(t)$ and $\tf(s)=f(s)$, it can be shown that
	$$
	|\tf(t)-\tf(m_k)|\leq |f(t)-f(m_k)|
	$$ 
	and
	$$
	|\tf(m_\ell')-\tf(s)|\leq |f(m_\ell')-f(s)|.
	$$ 
	The above inequalities, along with \eqref{eq:f_upper_Holder}, imply that
	$$
	|\tf(t)-\tf(s)|\leq C|t-m_k|^\alpha+C|m_k-m_\ell'|^\alpha+C|m_\ell'-s|^\alpha\leq 3C |t-s|^\alpha,
	$$ since $|t-m_k|$, $|m_k-m_\ell'|$, $|m_\ell'-s|$ are all less than $|t-s|$.
	
	In conclusion, $\tf$ is indeed an $\alpha$-H\"older continuous function, and consequently Theorem \ref{thm:spectrum_bdd_for_Holder} yields the upper bound $\dimAreg^{\theta}(\Graph(\tf))\leq\frac{2-\alpha-{\theta}}{1-{\theta}}$ for all $\theta\in (0,\alpha)$.
	
	Let $\theta\in (0,\alpha)$ and $\tilde{r}_k:=r_k^{\theta_0/\theta}$ so that $\tilde{r}_k^\theta=r_k^{\theta_0}$ for all $k\in \N$. This allows us to use the same sequence of squares $\{ Q_k \}$ we constructed for $\theta_0$. In order to show the lower bound $\dim_A^\theta(\Graph(\tf))\geq \frac{2-\alpha-\theta}{1-\theta}$, it suffices to prove that the number $N_k$ of squares of side-length $\tilde{r}_k$ of the form $[t-\tilde{r}_k^\theta+j\tilde{r}_k, t-\tilde{r}_k^\theta + (j+1) \tilde{r}_k] \times [\tf(t)-\tilde{r}_k^\theta+i\tilde{r}_k, \tf(t)-\tilde{r}_k^\theta + (i+1) \tilde{r}_k]$, $i, j\in \{0, 1, \dots, \lfloor r_k^{\theta-1}-1\rfloor\}$ needed to cover $R_k\cap\Graph(\tf) \subset Q_k\cap\Graph(\tf)$ is bounded below by a constant multiple of $\tilde{r}_k^{(\theta-1)\frac{2-\alpha-\theta}{1-\theta}}=\tilde{r}_k^{\alpha+\theta-2}$ for all $k\in \N$. That is because
	$$
	N_k\leq 9N(Q_k\cap \Graph(\tf),\tilde{r}_k),
	$$ 
	since any set of diameter at most $\tilde{r}_k$ cannot intersect more than $9$ such squares. The goal is to show for all $k\in \N$ a lower bound on the oscillation of $\tf$ on the intervals
	$$
	I_j^k:= [m_k-\tilde{r}_k^\theta+j\tilde{r}_k, m_k-\tilde{r}_k^\theta+(j+1)\tilde{r}_k] \subset \Proj_x(R_k),
	$$ 
	$j=0, 1, \dots, \lfloor \tilde{r}_k^{\theta-1}-1\rfloor$, which resembles \eqref{eq:f_lower_Holder}. This allows us to use a similar counting argument as in the proof of Theorem \ref{thm:spectrum_bdd_for_Holder}, but this time for the least number of subsquares needed, by counting the ones needed to cover $R_k\cap\Graph(\tf)$. Note the use of $\lfloor \tilde{r}_k^{\theta-1}-1\rfloor$ instead of $\lceil \tilde{r}_k^{\theta-1}-1\rceil$, which ensures that the interiors of the subsquares of side length $\tilde{r}_k$ that we use lie inside $R_k$, where $\Proj_x(\Graph(\tf)\cap R_k)=[m_k-\tilde{r}_k^\theta, m_k]$.
	
	Fix $k\in \N$ and set $r:=\tilde{r}_k$ and $I_j:=I_j^k$ for all $j$ to simplify the notation. We claim that
	\begin{equation}\label{eq:tf_lower_osc_int}
		\osc(\tf,I_j)\geq \tilde{c} r^\alpha
	\end{equation} for all $j=0, 1, \dots, \lfloor r^{\theta-1}-1\rfloor$, where $\tilde{c}=\min\{ 1, c/2 \}$ and $c$ is the constant from \eqref{eq:f_lower_Holder}. Indeed, by construction of $\tf$, for a given $I_j$ the part $\Graph(f;I_j)$ either never crosses any of the sides of $R_k$ during the $M_k+1$ reflections, in which case $\Graph(f;I_j)$ is isometric to $\Graph(\tf;I_j)$ and \eqref{eq:tf_lower_osc_int} trivially follows from \eqref{eq:f_lower_Holder}, or there is a minimal integer $i_0\in [0, M_k+1)$ for which $\Graph(f_{i_0,k};I_j)$ crosses and exits the lower or upper side of $R_k$. Note that
	$$
	\osc(f_{i_0,k},I_j)=\osc_{\text{in}}(f_{i_0,k},I_j)+\osc_{\text{out}}(f_{i_0,k},I_j)\leq 2 \max \{ \osc_{\text{in}}(f_{i_0,k},I_j),\osc_{\text{out}}(f_{i_0,k},I_j) \},
	$$ 
	where 
	$$
	\osc_{\text{in}}(f_{i_0,k},I_j):=\sup\{ |f_{i_0,k}(t)-f_{i_0,k}(s)|: \, \, t,\, s\in I_j \,\, \text{and} \,\, f_{i_0,k}(t),\, f_{i_0,k}(s)\in \Proj_y(R_k) \}
	$$ 
	is the oscillation of $f_{i_0,k}$ over $I_j$ inside $R_k$ and 
	$$
	\osc_{\text{out}}(f_{i_0,k},I_j):=\sup\{ |f_{i_0,k}(t)-f_{i_0,k}(s)|: \, \, t,\, s\in I_j \,\, \text{and} \,\, f_{i_0,k}(t),\, f_{i_0,k}(s)\in \Proj_y(R_k)^\complement \}
	$$ 
	is the oscillation of $f_{i_0,k}$ over $I_j$ outside $R_k$. See Figure \ref{fig:tf_lower_osc}.
	
	\begin{figure}[H]
		\centering
		\includegraphics[width=0.4\textwidth]{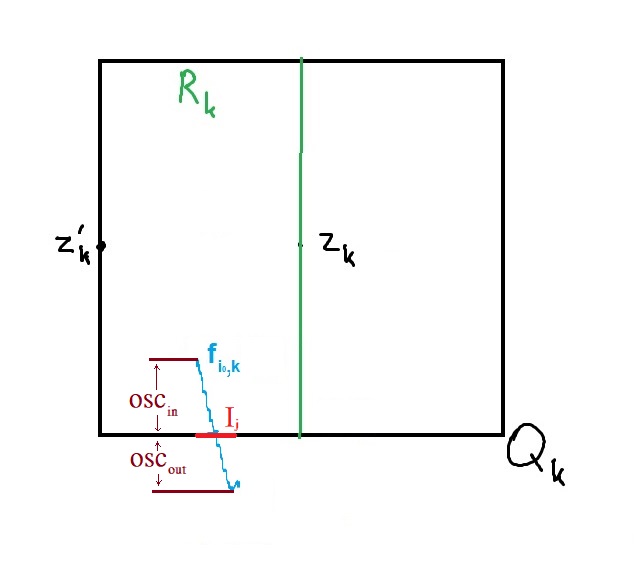}
		\caption{The oscillation of $f_{i_0,k}$ in $I_j$ after $i_0$ reflections of $\Graph(f)$, intersecting $Q_k$.}
		\label{fig:tf_lower_osc}
	\end{figure}
	
	As a result, at least one more reflection is required to construct $f_k$, during which either the part of $\Graph(f_{i_0,k})$ oscillating outside $R_k$ will be reflected inside $R_k$ and exit from the opposite side, in which case 
	\begin{equation}\label{eq:lower_osc_pf_1}
		\osc(f_k,I_j)=\osc_{\text{in}}(f_{i_0+1,k},I_j)=r^\theta\geq r^\alpha,
	\end{equation} 
	or it does not exit $R_k$, in which case
	\begin{equation}\label{eq:lower_osc_pf_2}
		\osc(f_k,I_j)=\max \{ \osc_{\text{in}}(f_{i_0,k},I_j),\osc_{\text{out}}(f_{i_0,k},I_j) \}\geq \osc(f_{i_0,k},I_j)/2.
	\end{equation} 
	However, after $i_0$ reflections of $\Graph(f)$, the part $\Graph(f_{i_0,k};I_j)$ stays isometric to $\Graph(f;I_j)$, which means that \eqref{eq:lower_osc_pf_2} implies
	$$
	\osc(f_k,I_j)\geq \osc(f,I_j)/2\geq c r^\alpha/2
	$$ 
	by \eqref{eq:f_lower_Holder}. Hence, the above inequality and \eqref{eq:lower_osc_pf_1} imply the desired oscillation condition \eqref{eq:tf_lower_osc_int} of $\tf$. 
	
	We now finish the proof by applying a counting argument similar to that in the proof of Theorem \ref{thm:spectrum_bdd_for_Holder}. Split all columns $J_j := I_j \times \R$ into squares of side-length $r$ of the form $[t-r^\theta+jr, t-r^\theta + (j+1) r] \times [\tf(t)-r^\theta+ir, \tf(t)-r^\theta + (i+1) r]$, $ i\in \Z$. In a given column $J_j$, by construction of $\tf$, since it is continuous and $\Graph(\tf;I_j)\subset Q_k$ for all $j$, there will be at least $ \osc(\tf,I_j)/r$ such squares needed to cover $R_k\cap \Graph(\tf)\subset Q_k\cap \Graph(\tf)$. Counting all these squares in columns $J_j$, $i=0, 1, \dots \lfloor r^{\theta-1}-1\rfloor$, we get that
	$$
	N_k \geq \sum_{i=0}^{\lfloor r^{\theta-1}-1\rfloor}\frac{\osc(\tf,I_i)}{r}.
	$$  
	Note that by \eqref{eq:tf_lower_osc_int}, every term of the sum on the right-hand side above is at least $\tilde{c}r^{\alpha-1}$.  Along with $\lfloor r^{\theta-1}\rfloor\geq r^{\theta-1}/2$, the above implies that
	$$
	r^{\theta-1} r^{\alpha-1} \lesssim N_k.
	$$Hence, by $N_k\leq 9N(Q_k\cap \Graph(\tf),r)$ we have shown
	$$
	r^{\alpha+\theta-2}\lesssim N(Q_k\cap \Graph(\tf),r)
	$$ 
	for arbitrary $k\in \N$, as needed to finish the proof. This concludes the proof of Theorem \ref{th:Holder_bd_sharp}. \qed
	
	\begin{remark}
		The function $\tf$ actually depends on the fixed arbitrary $\theta_0\in (0,\alpha)$ we pick in the beginning of the proof. However, that dependence is of no real importance, since $\tf$ is constructed once and satisfies \eqref{eq:Holder_bd_sharp} for all $\theta\in (0,\alpha)$, no matter what $\theta_0$ was initially picked.
	\end{remark}
	
	\begin{remark}
		The above proof in fact shows something stronger than just the indicated statement of Theorem \ref{th:Holder_bd_sharp}. Namely, it provides an algorithm that can be applied to any function $f$ satisfying both \eqref{eq:f_upper_Holder} and \eqref{eq:f_lower_Holder} which potentially increases the Assouad spectrum and dimension of its graph. For instance, while certain Takagi functions do not have graphs of full Assouad dimension (for instance $T_{a,2}$ for $a>1/2$, as studied in \cite{Kaenm_Takagi}), such functions can be modified within countably many squares by following the above process in order to increase the graph's Assouad spectrum (and also Assouad dimension).
	\end{remark}
	
	\begin{remark}
		By countable stability of the Hausdorff dimension (see \cite[p.~49]{FalconerBook}) and the construction of $\tf$, it follows that the Hausdorff dimension of $\Graph(\tf)$ coincides with that of $\Graph(f)$. As a result, the modifications in the proof of Theorem \ref{th:Holder_bd_sharp} can change some notions of dimension (e.g., the Assouad spectrum and dimension) and not others (e.g., the Hausdorff and upper box dimensions).
		
		Note that, at first glance, the reason why the upper box dimension of the graph does not change after the modifications in the proof of Theorem \ref{th:Holder_bd_sharp} might not be obvious, since the upper box dimension is not countably stable. However, a closer look at what these modifications imply for the covering number $N(D(\bz,1)\cap \Graph(\tf), r)$, for arbitrary $\bz\in \Graph(\tf)$ and $r>0$ reveals the reason. We invite the interested reader to fill in the details.
	\end{remark}
	
	While Theorems \ref{thm:spectrum_bdd_for_Holder} and \ref{th:Holder_bd_sharp} do not fully answer the open question \cite[Question 17.11.1]{FraserBook} that partially motivated this paper (namely, the precise value of the Assouad dimension of the graph of either the Weierstrass or Takagi functions), these theorems in turn raise a new question.
	
	\begin{question}\label{qu:WT_fit_in_squares}
		Do there exist choices of $a$ and $b$ with $a\in(0,1)$ and $b\in (1/a, \infty)$ such that either $\dimAreg^\theta \Graph(W_{a,b}) = \frac{2+\log_b(a)-\theta}{1-\theta}$ or $\dimAreg^\theta \Graph(T_{a,b}) = \frac{2+\log_b(a)-\theta}{1-\theta}$ for all relevant values of $\theta$?
	\end{question}
	
	In the context of Question \ref{qu:WT_fit_in_squares}, recall that both $W_{a,b}$ and $T_{a,b}$ are $\alpha$-H\"older continuous with $\alpha = -\log_b(a)$.
	
	\subsection{Lower bounds for the Assouad spectrum of co-H\"older graphs}\label{subsec:lower-bounds}
	
	If the answer to Question \ref{qu:WT_fit_in_squares} is negative, then no sequence of decreasing squares can be found which entirely contains the relevant part of the graph. One naturally wonders whether or not it would suffice to have only a sufficiently large part of the graph inside such squares (in order to obtain a nontrivial lower bound on the Assouad spectrum), or, alternatively, whether some flexibility is allowed in the lower H\"older oscillation condition. The following concept quantifies these ideas.
	
	\begin{definition}\label{def:lower-Holder}
		Let $f:I \to \R$ be defined on an interval $I \subset \R$. Let $\alpha\in(0,1)$ and $\eta,\eps\ge 0$ with $\eta+\eps < 1 - \alpha$. We say that $f$ has {\it uniform $\alpha$-co-H\"older estimates on large subintervals (with parameters $\eta$ and $\eps$)} if there exist a constant $c\in (0,1)$ and sequences $(R_n)$, $R_n<1$, $R_n \searrow 0$, and $(\bz_n)$, $\bz_n \in \Graph(f)$, so that the following two conditions hold true for the sequence of squares $(Q_n)$, $Q_n = Q(\bz_n,R_n)$:
		\begin{itemize}
			\item For each sub-interval $J \subset \Proj_x(Q_n)$,
			$$
			\osc(f,J) \ge c \, |\Proj_x(Q_n)|^\eta \, |J|^\alpha.
			$$
			\item There exist finitely many pair-wise disjoint intervals $I_k= I_k^n\subset \Proj_x(Q_n\cap \Graph(f))$, $k=1, \dots, M(n)$, with $\sum_{k=1}^{M(n)}|I_k|\geq c R_n^{1+\eps}$ and 
			\begin{equation}\label{eq: subset of proj condition}
				\min\{ 
				|I_k|: \, k=1,\dots, M(n) \}\geq c R_n^{1/\theta_0},
			\end{equation}where 
			\begin{equation}\label{eq:theta0}
				\theta_0=\frac{\alpha}{1-\eta-\eps}.
			\end{equation}
		\end{itemize}
	\end{definition}
	
	Informally, a map $f$ has uniform $\alpha$-co-H\"older estimates on large subintervals if the domain of $f$ within the $x$-projection of any square $Q_n$ is sufficiently large in a quantitative way, and on that domain $f$ expands distances by a fixed power factor, with coefficient which is allowed to decay (in a suitable fashion) in terms of the size of $Q_n$. The parameters $\eta$ and $\eps$ govern the relative size of the domain of $f$ and the decay rate of the lower H\"older constant. 
	
	\begin{remark}\label{rem:projections}
		Note that the condition $\sum_{k=1}^{M(n)}|I_k|\geq c R_n^{1+\eps}$ is not truly restrictive, as long as the weaker property $|\Proj_x (Q_n\cap \Graph(f))|\geq c R_n^{1+\eps}$ holds. The reason is we can write the interior of $\Proj_x (Q_n\cap \Graph(f))$ as a countable union of disjoint open intervals $\cup_{k\in \N} I_k$ with $|I_k|$ decreasing in $k$. Then there is finite $\tilde{M}>0$ such that $\sum_{k=1}^{\tilde{M}} |I_k|\geq |\Proj_x (Q_n\cap \Graph(f))|/2\geq c R_n^{1+\eps}/2$. The critical condition in Definition \ref{def:lower-Holder} is the existence of such a finite decomposition of $\Proj_x (Q_n\cap \Graph(f))$ with the property that all intervals have length bounded from below by $R_n^{1/\theta_0}$.
	\end{remark}
	
	We now state and prove the main theorem of this subsection.
	
	\begin{theorem}\label{th:assouad-spectrum-lower-bounds-with-parameters}
		If $f$ has uniform $\alpha$-co-H\"older estimates on large subintervals with parameters $\eta$, $\eps$, then
		$$
		\dimAreg^{\theta_0}\Graph(f)=\dim_A\Graph(f)=2,$$ where $\theta_0$ is as in \eqref{eq:theta0}. Moreover,
		\begin{equation}\label{eq:assouad-spectrum-lower-bounds-with-parameters}
			\dimAreg^\theta(\Graph(f)) \ge \frac{2-\alpha-(1+\eta+\eps)\theta}{1-\theta}
		\end{equation}
		for each $0<\theta\leq\theta_0$.
	\end{theorem}

	\begin{proof}
		Fix $\theta\in (0,\alpha(1-\eta-\eps)^{-1}]$ and
		\begin{equation}\label{eq:gamma-choice}
			\gamma < \frac{2 - \alpha - (1+\eta+\eps)\theta}{1-\theta}.
		\end{equation}
		The statements follow if we prove that $\dimAreg^\theta(\Graph(f)) > \gamma$.
		
		To this end,  let $(\bz_n)$, $(R_n)$, and $(Q_n)$ be as in Definition \ref{def:lower-Holder}, with $R_n < 1$. 
		Fix $n\in \N$ and using the intervals $I_k$ as in Definition \ref{def:lower-Holder} set 
		$$
		r_n:=\min\left\{ R_n^{1/\theta},  \min\{ 
		|I_k|: \, k=1,\dots, M(n) \}  \right\}.
		$$ 
		Since $R_n<1$ and $\theta\leq  \alpha(1-\eta-\eps)^{-1}=\theta_0$, we have that $R_n^{1/\theta}\leq R_n^{1/\theta_0}$. In view of \eqref{eq: subset of proj condition} we conclude
		\begin{equation}\label{eq: r_n for lower estimate}
			c R_n^{1/\theta}\leq r_n \leq R_n^{1/\theta},
		\end{equation} 
		and we use the sequence of scaling pairs $r_n,R_n$ to show the desired lower bound on the regularized Assouad spectrum of the graph. 
		
		To bound the regularized spectrum below by $\gamma$, it is enough to show that 
		$$
		N(Q_n\cap \Graph(f), r_n)\gtrsim\left( \frac{R_n}{r_n} \right)^\gamma.
		$$ 
		For any $j=1,\ldots, 2\ceil{R_n/r_n}$, a column of the form $C_j := [x_n-R_n+(j-1)r_n, x_n-R_n +jr_n]\times \R$ might, or might not intersect the graph of $f$ above an interval $I_k$. 
		Consider all of the non-empty intervals obtained as intersections $C_j \cap I_k$ for $k=1,\ldots,M(n)$ and $j=1,\ldots,2\ceil{R_n/r_n}$, and let $J_m$, $m = 1,\ldots,N$, be an enumeration of these intervals, ordered according to the natural order on the $x$-axis. Given such an interval $J_m$ we define $j_m$ and $k_m$ so that $J_m = C_{j_m} \cap J_{k_m}$. 
		Note that
		$$
		\sum_{m=1}^N |J_m| =\sum_{k=1}^{M(n)} |I_k|\geq c R_n^{1+\eps};
		$$
		since $|J_m|\leq r_n$ we conclude that
		\begin{equation}\label{eq: est on N}
			N\gtrsim \frac{R_n^{1+\eps}}{r_n}.
		\end{equation}
		Since $r_n \le |I_k|$ for each $k$, 
		%By choice of $r_n$, being smaller than the length of the smallest interval $I_k$, 
		we deduce that if $|J_m|<r_n/2$, then either
		$C_{j_m-1}$ or $C_{j_m+1}$ intersects $I_{k_m}$ and
		%$[x_n-R_n+(j_m-2)r_n, x_n-R_n +(j_m-1) r_n]$, or
		%$[x_n-R_n+j_m r_n, x_n-R_n +(j_m+1) r_n]$ intersects some $I_k$ and
		$$
		\max\{|J_{m-1}|, |J_{m+1}|\}\geq \frac{r_n}{2},
		$$
		where we adopt the convention that $J_{0}=J_{N+1}=\emptyset$. The preceding fact implies that the number of intervals $J_m$ with length at least $r_n/2$, denoted $\tilde{N}$, cannot be less than $N/6$. Fix the maximal sub-collection of intervals $J_m$ that have length at least $r_n/2$ and relabel them again with index $m$ to ease notation. Then, using only sub-squares of $Q_n$ of side-length $r_n$ that lie over such intervals $J_m$, $m=1,\dots, \tilde{N}$, to cover $Q_n\cap \Graph(f)$, we have that
		$$
		N(Q_n\cap \Graph(f),r_n)\geq \sum_{m=1}^{\tilde{N}} \frac{\osc(f,J_m)}{r_n}\gtrsim \sum_{m=1}^{\tilde{N}} \frac{R_n^\eta |J_m|^\alpha}{r_n}.
		$$ 
		But recall that $|J_m|\geq r_n/2$ and $\tilde{N}\geq N/6\gtrsim R_n^{1+\eps}/r_n$ by \eqref{eq: est on N}. Along with the above inequality, these imply that
		$$
		N(Q_n\cap \Graph(f),r_n)\gtrsim \frac{R_n^{1+\eta+\eps}}{r_n^{2-\alpha}}.
		$$ 
		To complete the proof we need to show that
		\begin{equation}\label{eq:end-of-the-proof}
			\frac{R_n^{1+\eta+\eps}}{r_n^{2-\alpha}}\gtrsim \left( \frac{R_n}{r_n} \right)^\gamma.
		\end{equation}
		For sufficiently large $n$, \eqref{eq:end-of-the-proof} follows from \eqref{eq:gamma-choice}, \eqref{eq: r_n for lower estimate}, and the bounds $\theta \le \theta_0$, $\eta+\eps < 1-\alpha$. This completes the proof of Theorem
		\ref{th:assouad-spectrum-lower-bounds-with-parameters}.
	\end{proof}
	
	\begin{remarks}\label{rem:assouad-spectrum-lower-bounds-with-parameters}
		\begin{itemize}
			\item[(i)] Suppose $f$ is a function as in Definition \ref{def:lower-Holder}, but with $\theta_0$ in \eqref{eq: subset of proj condition} replaced by some smaller value $\tilde{\theta}\in (0,\theta_0)$. A closer analysis of the proof of Theorem \ref{th:assouad-spectrum-lower-bounds-with-parameters} shows that the estimate
			$$
			\dimAreg^\theta (\Graph(f))\ge \frac{2-\alpha-(1+\eta+\eps)\theta}{1-\theta}
			$$ 
			then holds for all $\theta\in (0,\tilde{\theta})$.
			\item [(ii)] If $f$ is $\alpha$-H\"older and has uniform $\alpha$-co-H\"older estimates on large subintervals with parameters $\eta = \eps = 0$, then the lower bound in \eqref{eq:assouad-spectrum-lower-bounds-with-parameters} coincides with the upper bound in Theorem \ref{thm:spectrum_bdd_for_Holder} and we have a formula for the regularized Assouad spectrum of $\Graph(f)$. This case corresponds to the situation in which $f$ has the lower oscillation estimate \eqref{eq:f_lower_Holder} and a definite proportion of the graph of $f$ over $\Proj_x(Q_n)$ lies within $Q_n$, when measured with respect to the length measure on $\Proj_x(Q_n)$ in the quantitative way outlined in Definition \ref{def:lower-Holder}.
			\item[(iii)] The restriction $\eta+\eps<1-\alpha$ implies that the lower bound in \eqref{eq:assouad-spectrum-lower-bounds-with-parameters} strictly exceeds $2-\alpha$. If $f$ is also $\alpha$-H\"older, the oscillation condition of Definition \ref{def:lower-Holder} is enough to ensure that $\ovdimB\Graph(f)=2-\alpha$. In this case, the lower bound on the regularized spectrum of $\Graph(f)$ is non-trivial (i.e., exceeds the box dimension of $\Graph(f)$) for all $\theta>0$. 
			\item [(iv)] The Takagi functions studied in \cite{Kaenm_Takagi}, namely $T_{a,2}$ for $a>1/2$, do not satisfy uniform co-H\"older estimates as in Definition \ref{def:lower-Holder} with parameters $\eta$ and $\eps$ such that $\eta+\eps < 1 - \alpha = 1 + \log_2(a)$. Indeed, if they did, then their graphs would have Assouad dimension equal to $2$ by Theorem \ref{th:assouad-spectrum-lower-bounds-with-parameters}, which was shown not to be the case in \cite{Kaenm_Takagi}.
		\end{itemize}
	\end{remarks}
	
	We emphasize that Remark \ref{rem:assouad-spectrum-lower-bounds-with-parameters}(iv) does not prevent other Takagi or Weierstrass functions from satisfying such estimates, and especially for $T_{a,b}$ with $b$ non-integer. While it is known that there exist choices of $a$, $b$ and $\theta$ for which $\ovdimB \Graph(T_{a,b})<\dim_A^\theta \Graph(T_{a,b})$ (see  \cite[Remark 1.2]{TakagiGraphYu}), the following weaker version of Question \ref{qu:WT_fit_in_squares} remains open.
	
	\begin{question}\label{Que: weak lower bound T and W}
		Does there exist any choice of $a$ and $b$ so that either $T_{a,b}$ or $W_{a,b}$ has graph with (regularized) Assouad spectrum satisfying the lower bound \eqref{eq:assouad-spectrum-lower-bounds-with-parameters}?
	\end{question}
	
	There are other examples of H\"older functions that could potentially satisfy uniform co-H\"older estimates. Hildebrand and the first author \cite{ChronAJi} study a class of fractal Fourier series of the form $F(t):=\sum_{n=1}^\infty \frac{f(n) e^{2\pi i n^k t}}{n^p}:[0,1]\rightarrow \C$ for positive $p>0$, $k\in \N$, where $f:\N\rightarrow \C$ satisfies a certain exponential sum growth condition. These are H\"older functions whose real and imaginary part have graphs that exhibit fractal behavior. In fact, Weierstrass functions belong in this class (as the real part of such series) and establish the sharpness of the box dimension bounds proved for the graphs of these fractal Fourier series (see \cite{ChronAJi}). It is possible that suitable conditions on the coefficients $f(n)$ may yield further examples of graphs satisfying \eqref{eq:Holder_bd_sharp} or \eqref{eq:assouad-spectrum-lower-bounds-with-parameters}.
	
	\section{Dimensions of Sobolev graphs}\label{sec:Sobolev}
	
	In this section we estimate the Assouad spectra of graphs of Sobolev functions. In particular, we prove Theorems \ref{th:ass-spectrum-sobolev} and \ref{th:ass-spectrum-sobolev-2}.
	
	First, we remark that box counting dimensions of graphs of Sobolev (and more generally, BV) functions are well understood. In fact, the following is a known result.
	
	\begin{theorem}\label{th:ubd-bv}
		Let $f \in BV(I)$. Then $\ubdim(\Graph(f)) = 1$.
	\end{theorem}
	
	We sketch a proof of Theorem \ref{th:ubd-bv}.\footnote{While we did not locate a published version of this proof in the literature, it can be found at {\tt https://mathoverflow.net/questions/304573/hausdorff-dimension-of-the-graph-of-an-increasing-function}. For an alternate proof, see {\tt https://mathoverflow.net/questions/327698/hausdorff-dimension-of-the-graph-of-} {\tt a-bv-function-in-1-dimensional-setting}.} Recall that $f \in \BV(I)$ if and only if $f=g-h$ for two monotone functions $g$ and $h$ on $I$. Thus the first step of the proof of Theorem \ref{th:ubd-bv} is to show that $\ubdim(\Graph(g)) = 1$ for any monotone function $g:I \to \R$. To see why this is true, observe that a clockwise rotation of $\R^2$ by angle $\pi/4$ exhibits the graph of $g$ as geometrically congruent to the graph of a $1$-Lipschitz function $h$ defined on a new interval $I'$. Since $\Graph(h) = H(I')$ for the function $H(x) := (x,h(x))$, and $H$ is a bi-Lipschitz embedding of $I'$ into $\R^2$, we conclude that
	$\ubdim(\Graph(g)) = \ubdim(\Graph(h)) = \ubdim(H(I')) = \ubdim(I') = 1$ as desired.
	
	\begin{remark}
		Observe that the same argument in fact shows that $\Graph(g)$ has Assouad dimension (and all Assouad spectra) equal to $1$ whenever $g$ is monotone.
	\end{remark}
	
	The next step in the proof is to relate the dimension of the graph of a sum of two functions to the dimensions of the graphs of the addends.
	
	\begin{lemma}\label{lem:ubd-sum}
		$\ubdim(\Graph(g+h)) \le \max\{ \ubdim(\Graph(g)), \ubdim(\Graph(h)) \}$ for functions $g,h:I \to \R$.
	\end{lemma}
	
	We call a notion of dimension $D$  {\it stable with respect to graph sums} if $$D(\Graph(g+h)) \le \max \{D(\Graph(g)),D(\Graph(h))\},$$ where $g$ and $h$ are continuous functions defined on a common interval $I$. Thus Lemma \ref{lem:ubd-sum} asserts that box-counting dimension is stable with respect to graph sums.
	
	For a proof of Lemma \ref{lem:ubd-sum}, see e.g.\ \cite[Lemma 2.1]{falconer-fraser}. Another approach\footnote{We learned of this proof at {\tt https://mathoverflow.net/questions/331714/hausdorff-dimension-of-the-graph-} {\tt of-the-sum-of-two-continuous-functions}.} is to establish the identity
	\begin{equation}\label{eq:ubdimGraph-g-equivalence}
		{\footnotesize \ubdim(\Graph(g)) = \inf \left\{ p>0 \, : \, {\mbox{$\exists\,C>0$ such that $\forall\,\delta>0$ there is a partition $I = I_1\cup\cdots\cup I_N$} \atop \mbox{into subintervals with $N \le C\delta^{-p}$, $|I_j| \le\delta$, and $\osc(g,I_j) \le \delta$ for all $j$}} \right\},}
	\end{equation}\normalsize
	after which the desired conclusion follows from the subadditivity of oscillation.
	
	\begin{remark}
		The preceding argument does not apply to either the Assouad spectra or the Assouad dimension, since a corresponding formula analogous to \eqref{eq:ubdimGraph-g-equivalence} does not hold. We therefore cannot conclude by such methods either that these Assouad-type dimensions are stable with respect to graph sums, nor that the dimensions of BV graphs are equal to one. In fact, we will argue in the opposite direction as a consequence of Theorem \ref{th:ass-spectrum-sobolev-2} and deduce that these Assouad-type dimensions are not stable with respect to graph sums. See Corollary \ref{cor:dim_A_theta_not_graph_sum_stable}.
	\end{remark}
	
	With an eye towards the proof of Theorem \ref{th:ass-spectrum-sobolev-2}, we introduce an example.\footnote{We take this opportunity to acknowledge that a similar construction was previously known to Fraser as a counterexample to the stability of Assouad dimension under graph sums. In Remark \ref{rem:failure-stability-assouad-dimension-graphs} we recover the latter observation by a slightly different line of reasoning.}
	
	\begin{example}\label{ex:w1pex}
		Fix a decreasing sequence $(a_m)$ with $a_m \searrow 0$, and define $f:[0,a_1] \to \R$ as follows: $f(0) = 0$ and
		\begin{equation}\label{eq:f-defn}
			f(x) = \begin{cases} - a_{2\ell} + \frac{a_{2\ell} + a_{2\ell+1}}{a_{2\ell} - a_{2\ell +1}} (a_{2\ell}-x), &\mbox{if $x\in[a_{m+1},a_m]$ and $m=2\ell$ is even,} \\ a_{2\ell-1} - \frac{a_{2\ell-1} + a_{2\ell}}{a_{2\ell-1} - a_{2\ell}} (a_{2\ell-1}-x), &\mbox{if $x\in[a_{m+1},a_m]$ and $m=2\ell-1$ is odd.}
			\end{cases}
		\end{equation}
		Observe that the graph of $f$ is a piecewise linear curve consisting of line segments joining the points
		\begin{equation}\label{eq:pm}
			\bz_m := (a_m, (-1)^m a_m), \qquad m=1,2,\ldots
		\end{equation}
		in succession.
		
		We are interested in characterizing when $f$ lies in a suitable Sobolev space $W^{1,p}$, $1\le p\le\infty$. Towards this end, note that 
		$$
		|f'(x)| = \sum_{m=1}^\infty \frac{a_{m}+a_{m+1}}{a_{m}-a_{m+1}} \chi_{\{x:a_{m+1}<x<a_{m}\}},
		$$
		and hence the {\it $p$-energy} of $f$ is
		$$
		E_p(f)^p := \int_{[0,a_1]} |f'(x)|^p \, dx = \sum_{m=1}^\infty \frac{(a_{m}+a_{m+1})^p}{(a_{m}-a_{m+1})^{p-1}}.
		$$
		Thus the membership of $f$ in a specific Sobolev class $W^{1,p}$ depends on the relative rates of convergence to zero of $(a_m)$ and $(\eps_m)$, where
		$$
		\eps_m := a_{m} - a_{m+1}.
		$$
		In the proof of Theorem \ref{th:ass-spectrum-sobolev-2} we will focus on a specific example, for which we will compute the sharp range of Sobolev membership and the precise Assouad spectrum of the graph.
	\end{example}
	
	We leave the proof of the following assertion as an exercise for the reader.
	
	\begin{claim}\label{claim-assouad-dimension}
		If $a_{m+1}/a_m \to 1$ as $m\to \infty$, then $\Graph(f)$ has a full quarter-plane of $\R^2$ as a weak tangent (when considering blow-ups centered at the origin $(0,0)$). Hence, in this situation, we have
		\begin{equation}\label{eq:assouad-dimension-equals-2}
			\dim_A(\Graph(f)) = 2;
		\end{equation}
		see \cite[Theorem 5.1.3]{FraserBook} for a proof of the fact that the Assouad dimension of a set $E$ is equal to the supremum of the Hausdorff dimensions of all of the weak tangents of $E$.
	\end{claim}
	
	\begin{remark}\label{rem:failure-stability-assouad-dimension-graphs}
		Note that it already follows from \eqref{eq:assouad-dimension-equals-2} that Assouad dimension is not stable with respect to graph sums. Indeed, there exist functions $f$ of the form in \eqref{eq:f-defn} which lie in the BV class; then $f=g-h$ for monotone $g$ and $h$ but $\assdim(\Graph(f))=2$ while $\assdim(\Graph(g)) = \assdim(\Graph(h)) = 1$.
	\end{remark}
	
	Applying the argument from Remark \ref{rem:failure-stability-assouad-dimension-graphs}, but now in the context of Theorem \ref{th:ass-spectrum-sobolev-2}, yields the following corollary.
	
	\begin{corollary}\label{cor:dim_A_theta_not_graph_sum_stable}
		For any $0<\theta<1$, the Assouad spectrum $\dim_A^\theta$ is not stable with respect to graph sums.
	\end{corollary}
	
	We now prove Theorems \ref{th:ass-spectrum-sobolev} and \ref{th:ass-spectrum-sobolev-2}, which provide sharp upper bounds for the Assouad spectra of Sobolev graphs.
	
	\begin{proof}[Proof of Theorem \ref{th:ass-spectrum-sobolev}]
		Give $f \in W^{1,p}(I)$, $1<p<\infty$, we fix a subinterval $J \subset I$ and let $x,y \in J$. Then
		$$
		f(x) - f(y) = \int_x^y f'(t) \, dt
		$$
		and so
		$$
		|f(x)-f(y)| \le \int_{[x,y]} |f'(t)| \, dt \le |x-y|^{1-1/p} \, ( \int_J |f'|^p )^{1/p}
		$$
		whence
		$$
		\osc(f,J) \le |J|^{1-1/p} \, (\int_J |f'|^p)^{1/p}.
		$$
		When $p=1$ we continue to get the inequality $\osc(f,J) \le \int_J |f'|$, and when $p=\infty$ we continue to get the inequality $\osc(f,J) \le |J| \text{ess sup}|f'|$.
		
		Now fix $0<\theta<1$. We will estimate $\dimAreg^\theta(\Graph(f))$ from above. Fix $x_0 \in I$ and $0<R\le 1$. We denote by $\bz_0 = (x_0,f(x_0))$
		the corresponding point in $\Graph(f)$. Let $Q_0$ be the square with center $\bz_0$ and side length $2R$, and let $I_0 = [x_0-R,x_0+R]$ be the projection of $Q_0$ to the $x$-axis.
		
		Given $0<r\le R^{1/\theta}$ we can cover $I_0$ with $M:= \lceil \tfrac{2R}{r} \rceil$ intervals $J_k$, each of length $r$. Note that
		$$
		\frac{2R}{r} \le M \le \frac{2R}{r} + 1 \le \frac{3R}{r}.
		$$
		Then, noting that $\sum_{k=1}^M (\int_{J_k} |f'|^p)^{1/p}\leq M^{1-1/p} (\sum_{k=1}^M\int_{J_k} |f'|^p)^{1/p}$ by H\"older's inequality, we have
		\begin{equation*}\begin{split}
				N(Q_0 \cap \Graph(f),r) &\le \sum_{k=1}^M \left( \frac{\osc(f,J_k)}{r} + 2 \right) \\
				&\le \sum_{k=1}^M \left( 2 + \left(\Barint_{J_k} |f'|^p\right)^{1/p} \right) \\
				&\le 2M + M^{1-1/p} r^{-1/p} \left(  \sum_{k=1}^M \int_{J_k} |f'|^p \right)^{1/p} \\
				&\leq 2M + M^{1-1/p} r^{-1/p} \left( \int_{I} |f'|^p \right)^{1/p} \\
				&\lesssim \frac{R}{r} + \frac{R^{1-1/p}}{r} ||f'||_{p,I},
		\end{split}\end{equation*}
		where $||h||_{p,J} := (\int_J |h|^p)^{1/p}$. Since $R\le 1$ we in turn conclude that
		$$
		N(Q_0 \cap \Graph(f),r) \lesssim R^{1-1/p} r^{-1} ( 1 + ||f'||_{p,I} ) .
		$$
		Next, the restriction $r\le R^{1/\theta}$ implies
		$$
		R^{-1/p} \le \left( \frac{R}{r} \right)^{\tfrac{\theta}{(1-\theta)p}}
		$$
		and hence,
        we have
		$$
		N(Q_0 \cap \Graph(f),r) \lesssim \left( \frac{R}{r} \right)^{1+\tfrac{\theta}{(1-\theta)p}}
		$$
		where the implicit constant depends on $||f'||_{p,I}$. Since this conclusion holds true for all $0<r\le R^{1/\theta} < R \le 1$ and all $\bz_0 \in \Graph(f)$, we conclude that \eqref{eq:ass-spectrum-sobolev} holds as desired.
	\end{proof}
	
	\begin{proof}[Proof of  Theorem \ref{th:ass-spectrum-sobolev-2}]
		Fix $p>1$. For pedagogical reasons, we first construct a function $f$ so that $f \in W^{1,q}(I)$ for all $1\le q<p$ and $\dim_A^\theta(\Graph(f)) = 1 + \theta/((1-\theta)p)$. At the end of the proof, we will indicate how to modify the construction to cover the borderline case $q=p$.
		
		We return to the situation of Example \ref{ex:w1pex} and specialize to a one-parameter family of examples. For $s>2$, let
		\begin{equation}\label{eq:am}
			a_m = m^{1-s}.
		\end{equation}
		Then
		$$
		\eps_m = a_m - a_{m+1} = (s-1)(m+\delta_m)^{-s} \qquad \mbox{for some $0<\delta_m<1$}
		$$
		by the Mean Value Theorem. 
		
		\medskip
		
		We choose $s:=p+1$ and consider the points $\{ \bz_m \}$ defined as in \eqref{eq:pm} and the function $f=f_s:[0,a_1] \to \R$ defined in \eqref{eq:f-defn}. The $q$-energy is
		$$
		E_q(f_s) = \sum_{m=1}^\infty \frac{(m^{1-s} + (m+1)^{1-s})^q}{(m^{-s} - (m+1)^{-s})^{q-1}} \simeq \sum_{m=1}^\infty \frac{m^{(1-s)q}}{m^{-s(q-1)}}  \, ,
		$$
		where the notation $\simeq$ means that the sums on either side are simultaneously finite or infinite. Hence $E_q(f_s) < \infty$ if and only if $q-s<-1$, so
		$$
		f_s \in W^{1,q}([0,a_1]) \qquad \mbox{for all $q<s-1=p$,}
		$$
		and
		$$
		f_s \not \in W^{1,s-1}([0,a_1]).
		$$
		Applying Theorem \ref{th:ass-spectrum-sobolev} and letting $q \to s-1$ yields
		$$
		\assdim^\theta(\Graph(f_s)) \le \dimAreg^\theta (\Graph(f_s))\leq 1 + \frac{\theta}{(1-\theta)(s-1)} \, .
		$$
		Now fix $\theta<(s-1)/s$ and 
		\begin{equation}\label{eq:gamma-bounding}
			\gamma < 1 + \frac{\theta}{(1-\theta)(s-1)} \, ,
		\end{equation}
		We will show that $\assdim^\theta(\Graph(f_s)) \ge \gamma$.
		
		It suffices to find $c>0$ and $\eps>0$ and sequences $(r_n)$ and $(R_n)$ with $0<r_n = R_n^{1/\theta} < R_n \le 1$ so that
		\begin{equation}\label{eq:covering-lower-bound}
			N(Q(0,R_n)\cap \Graph(f_s),r_n) \ge c n^\eps \left( \frac{R_n}{r_n} \right)^\gamma.
		\end{equation}
		In order to obtain \eqref{eq:covering-lower-bound} we will exhibit a large collection $\cD$ of points in $Q(0,R_n) \cap \Graph(f_s)$ with pairwise distances all $>r_n$. See \cite[Section 1.2]{FraserBook} or \cite[Lemma 3.4]{luu:assouad} for more information about the relationship between packing and covering numbers.
		
		We choose the sequence $(R_n)$ to be $R_n = a_n = n^{1-s}$ for all $n \ge n_0$ for a sufficiently large integer $n_0$ (to be determined later). Observe that
		$$
		\bz_m \in Q(0,R_n) \, \Longleftrightarrow \, a_m \le R_n \, \Longleftrightarrow \, m \ge n.
		$$
		Fix a large even integer $N(n) > n$ (to be determined later). The desired set $\cD$ consists of a finite collection of equally spaced collinear points along the line segments $[\bz_m,\bz_{m+1}]$, $[\bz_{m+2},\bz_{m+3}]$, \dots, $[\bz_{n+N(n)-2},\bz_{n+N(n)-1}]$. Along each line segment $[\bz_m,\bz_{m+1}]$, the elements of $\cD$ are equally spaced points at distance approximately $r_n$.
		
		More precisely, let
		\begin{equation}\label{eq:D}
			\cD = \{ \bw_{m,k} \, : \, m = n,n+2,n+4,\ldots,n+N(n)-2, k=1,2,\ldots,M(n) \},
		\end{equation}
		where, for each $m$, the points $\{\bw_{m,k}:k=1,\ldots,M(m)\}$ constitute a maximal $r_n$-separated set in the line segment $[\bz_m,\bz_{m+1}]$. For each $k=1,\ldots,M(m)-1$, we have
		$$
		|\bw_{m,k}-\bw_{m,k+1}| = \frac{|\bz_m-\bz_{m+1}|}{M(m)}
		$$
		and
		\begin{equation}\label{eq:Mmbounds}
			\frac{|\bz_m-\bz_{m+1}|}{M(m)-1} > r_n \ge  \frac{|\bz_m-\bz_{m+1}|}{M(m)} .
		\end{equation}
		Consequently,
		\begin{equation}\label{eq:Mm}
			M(m) = \left \lceil \frac{|\bz_m-\bz_{m+1}|}{r_n} \right \rceil \geq \left \lceil \frac{\sqrt{2}(a_m-a_{m+1})}{r_n} \right \rceil = \left \lceil \sqrt{2} (s-1) n^{(s-1)/\theta} (m+\delta_m)^{-s} \right \rceil.
		\end{equation}
		We impose the condition that $M(m) \ge 2$ for each $m=n,n+2,n+4,\ldots,n+N(n)-2$; this leads to an upper bound on the choice of $N(n)$. In fact, the condition needs only hold for $m=n+N(n)-2$ and we observe that
		$$
		M(n+N(n)-2) \ge 2
		$$
		holds true provided
		$$
		\sqrt{2} (s-1) n^{(s-1)/\theta} (n+N(n)-2+\delta_{n+N(n)-2})^{-s} \ge 2
		$$
		which holds if and only if
		\begin{equation}\label{eq:Nn1}
			N(n) \le \left( \frac{s-1}{\sqrt{2}} \right)^{1/s} n^{(s-1)/(s\theta)} - (n-2) - \delta_{n+N(n)-2}.
		\end{equation}
		We also impose the condition that $|\bw_{m,k} - \bw_{m',k'}| > r_n$ whenever $m \ne m'$ and $k\in\{1,\ldots,M(m)\}$, $k'\in\{1,\ldots,M(m')\}$. Assume without loss of generality that $m<m'$. Then $m\le n+N(n)-4$ and $m'\le n+N(n)-2$ and hence
		\begin{equation*}\begin{split}
				|\bw_{m,k} - \bw_{m',k'}| &\ge a_{m+1}-a_{m'}  \ge a_{n+N(n)-3} - a_{n+N(n)-2} \\
				&= \eps_{n+N(n)-3} \\
				&= (s-1)(n+N(n)-3 + \delta_{n+N(n)-3})^{-s}.
		\end{split}\end{equation*}
		Thus the desired condition holds provided
		$$
		(s-1)(n+N(n)-3 + \delta_{n+N(n)-3})^{-s} > r_n = n^{(1-s)/\theta}
		$$
		which in turn holds if and only if
		\begin{equation}\label{eq:Nn2}
			N(n) < (s-1)^{1/s} \, n^{(s-1)/(s\theta)} - (n-3) - \delta_{n+N(n)-3}.
		\end{equation}
		The following claim holds in view of the assumption $\theta<(s-1)/s$.
		
		\begin{claim}\label{claim1}
			There exist a constant $c_0(s)>0$ and a large integer $n_0=n_0(s)$ so that, for all $n \ge n_0$, both \eqref{eq:Nn1} and \eqref{eq:Nn2} are satisfied when $N(n)$ is chosen to be the largest even integer less than or equal to $c_0(s) n^{(s-1)/(s\theta)}$. 
		\end{claim}
		
		For the remainder of the proof, we consider only integers $n \ge n_0$, where $n_0$ is chosen as in the preceding claim. 
		
		By the construction and the above properties, the set $\cD$ defined in \eqref{eq:D} with this choice of $N(n)$ and integers $M(m)$ as in \eqref{eq:Mm} has the property that
		$|\bw-\bw'|>r_n$ whenever $\bw$ and $\bw'$ are distinct points in $\cD$. To finish the proof, we estimate the cardinality $\# \cD$ of $\cD$ from below. Note that
		$$
		\# \cD = M(n) + M(n+2) + \cdots + M(n+N(n)-2)
		$$
		and, by \eqref{eq:Mmbounds},
		$$
		\# \cD \ge \frac{|\bz_n-\bz_{n+1}| + |\bz_{n+2} - \bz_{n+3}| + \cdots + |\bz_{n+N(n)-2} - \bz_{n+N(n)-1}|}{r_n}.
		$$
		By the choice of the points $\bz_m=(a_m,(-1)^m a_m)$,
		$$
		|\bz_m-\bz_{m+1}| = \sqrt{(a_m-a_{m+1})^2 + (a_m+a_{m+1})^2} = \sqrt{2} \, \sqrt{a_m^2 + a_{m+1}^2} \ge a_m + a_{m+1}.
		$$
		Thus
		\begin{equation*}\begin{split}
				\# \cD 
				&\ge \frac{a_n+a_{n+1}+\cdots+a_{n+N(n)-1}}{r_n} \\
				&= n^{(s-1)/\theta} \sum_{m=n}^{n+N(n)-1} m^{1-s} \\
				&\ge c_1(s) n^{(s-1)/\theta} (n^{2-s} - (n+N(n)-1)^{2-s}) \\
				&\ge c_2(s) n^{(s-1)/\theta+2-s},
		\end{split}\end{equation*}
        where $c_1(s)$ depends on the integral test constant, and $c_2(s)$ depends on $c_0(s), c_1(s),$ and the domination constant of $n^{2-s}$ over $(n+n^{\tfrac{s-1}{s\theta}}-1)^{2-s}$.
		Note also that
		$$
		n^\eps (R_n/r_n)^\gamma = n^{\eps + (s-1)(\tfrac1\theta-1)\gamma}.
		$$
		In order to verify that \eqref{eq:covering-lower-bound} holds, it suffices to show that
		\begin{equation}\label{eq:final-estimate}
			(s-1)(\frac1\theta-1)\gamma < \frac{s-1}{\theta} + 2-s;
		\end{equation}
		then a suitable $\eps = \eps(s,\gamma,\theta)$ can be chosen so that
		$$
		\eps + (s-1)(\frac1\theta-1)\gamma < \frac{s-1}{\theta} + 2-s.
		$$
		The inequality in \eqref{eq:final-estimate} is the same as our assumed bound \eqref{eq:gamma-bounding}.
		
		To extend to the borderline case, we introduce a logarithmic factor. We replace the choice of $(a_m)$ in \eqref{eq:am} with the following:
		\begin{equation}\label{eq:am2}
			a_m := m^{1-s} \log^{-2}(m), \qquad s>2, m \ge 2.
		\end{equation}
		In this case, we find (again by the Mean Value Theorem) that
		$$
		\eps_m = a_m - a_{m+1} = \left. (s-1) x^{-s} \log^{-2}(x)\bigl( 1 + \tfrac{2}{\log(x)} \bigr) \right|_{x=m+\delta_m}
		$$
		for some $0<\delta_m<1$. Setting $s=p+1$ as before, we let $f = f_s:[0,a_1] \to \R$ be the function defined in \eqref{eq:f-defn}. The $p$-energy of $f$ is
		$$
		E_p(f_s) \simeq \sum_{m=2}^\infty m^{-1} \log^{-2}(m) < \infty.
		$$
		Thus $f_s \in W^{1,p}([0,a_1])$. It remains to show that the Assouad spectrum formula
		$$
		\assdim^\theta(\Graph(f_s)) = 1 + \frac{\theta}{(1-\theta)(s-1)}
		$$
		continues to hold. As before, only the lower bound needs to be verified, so we fix $\gamma$ as in \eqref{eq:gamma-bounding} and show that $\assdim^\theta(\Graph(f_s)) > \gamma$. The argument is similar to that described in detail in the first part of the proof. With the choice $R_n = a_n$, $n\ge 2$, and $r_n = R_n^{1/\theta}$ as before, we consider $M(m) = \lceil \tfrac{\sqrt{2} \eps_m}{r_n} \rceil$. The condition $M(m) \ge 2$ for all $m=n, n+2,\ldots,n+N(n)-2$, reduces to
		\begin{equation}\label{eq:mn1}
			M(n+N(n)-2) \ge 2.
		\end{equation}
		Define
		$$
		\eta_n := \left.\left.\left( (1+\tfrac{2}{\log(x)})^{1/s} - 1 \right)\right|_{x=m+\delta_m}\right|_{m=n+N(n)-2}
		$$ 
		and note that the condition $n \ge 2$ implies $\eta_n < 1$. The desired condition \eqref{eq:mn1} is implied by
		\begin{equation}\label{eq:2a}
			\left. \left. x^s \log^2(x) \right|_{x=m+\delta_m} \right|_{m=n+N(n)-2} \le \frac{s-1}{\sqrt{2}} n^{(s-1)/\theta} \log^{2/\theta}(n) (1+\eta_n)^s.
		\end{equation}
		Fixing the Orlicz function $\Phi(t) = t\,\log^{2/s}(t)$ and setting $c(s) = ((s-1)/\sqrt{2})^{1/s}$ we rewrite \eqref{eq:2a} as
		\begin{equation}\label{eq:2b}
			\left. \left. \Phi(x) \right|_{x=m+\delta_m} \right|_{m=n+N(n)-2} \le c(s) n^{(s-1)/(s\theta)} \log^{2/(s\theta)}(n) (1+\eta_n),
		\end{equation}
		or equivalently,
		\begin{equation}\label{eq:2c}
			\left. (m+\delta_m) \right|_{m=n+N(n)-2} \le \Phi^{-1} \left( c(s) n^{(s-1)/(s\theta)} \log^{2/(s\theta)}(n) (1+\eta_n) \right).
		\end{equation}
		Inequality \eqref{eq:2c} translates to the following constraint on the choice of $N(n)$:
		\begin{equation}\label{eq:2d}
			N(n) \le \Phi^{-1} \left( c(s) n^{(s-1)/(s\theta)} \log^{2/(s\theta)}(n) (1+\eta_n) \right) - (n-2) - \delta_{n+N(n)-2}.
		\end{equation}
		Recall that both $\eta_n$, $n \ge 2$, and $\delta_m$, $m \ge 2$, lie in the interval $(0,1)$. Furthermore, the assumption $\theta < \tfrac{p}{p+1} = \tfrac{s-1}{s}$ implies that $(s-1)/(s\theta)>1$. For large values of the argument $t$, we have 
		$$
		\Phi^{-1}(t) \simeq t \log^{-2/s}(t)(1+o(1))
		$$
		and hence \eqref{eq:2d} translates to
		\begin{equation}\label{eq:2e}
			N(n) \le c'(s) n^{(s-1)/(s\theta)} \log^{\tfrac2s(\tfrac1\theta-1)}(n) (1+\eta_n) (1+o(1))
		\end{equation}
		where $c'(s) = c(s)(s\theta/(s-1))^{2/s}$.
		
		In a similar fashion, we analyze the previous second condition
		\begin{equation}\label{eq:mn2}
			\eps_{n+N(n)-3} > r_n.
		\end{equation}
		We now define
		$$
		\overline\eta_n := \left.\left.\left( (1+\tfrac{2}{\log(x)})^{1/s} - 1 \right)\right|_{x=m+\delta_m}\right|_{m=n+N(n)-3}
		$$
		and we again note that $\overline\eta_n < 1$ since $n \ge 2$. After a similar chain of reasoning, we deduce that \eqref{eq:mn2} is implied by the following constraint on $N(n)$:
		\begin{equation}\label{eq:3a}
			N(n) \le \Phi^{-1} \left( \overline{c}(s) n^{(s-1)/(s\theta)} \log^{2/(s\theta)}(n) (1+\overline\eta_n) \right) - (n-3) - \delta_{n+N(n)-3}.
		\end{equation}
		Here $\overline{c}(s) = (s-1)^{1/s}$. Using the previous asymptotics for the inverse of the Orlicz function $\Phi$, we reduce \eqref{eq:3a} to
		\begin{equation}\label{eq:3b}
			N(n) \le \overline{c}'(s) n^{(s-1)/(s\theta)} \log^{\tfrac2s(\tfrac1\theta-1)}(n) (1+\overline\eta_n) (1 + o(1))
		\end{equation}
		where $\overline{c}'(s) = \overline{c}(s)(s\theta/(s-1))^{2/s}$. Our previous Claim \ref{claim1} is now replaced by
		
		\begin{claim}\label{claim2}
			There exist a constant $\overline{c}_0(s)>0$ and a large integer $\overline{n}_0=\overline{n}_0(s)$ so that, for all $n \ge \overline{n}_0$, both \eqref{eq:2e} and \eqref{eq:3b} are satisfied when $N(n)$ is chosen to be the largest even integer less than or equal to $\overline{c}_0(s) n^{(s-1)/(s\theta)}$. 
		\end{claim}
		
		The final part of the proof is exactly the same as in the previous case and is left to the reader. The desired lower bound for the cardinality $\#\cD$ of the set defined in \eqref{eq:D} holds due to the choice of $\gamma$ in \eqref{eq:gamma-bounding}. Note that additional logarithmic factors show up in the relevant inequalities, however, these are irrelevant to the desired conclusion as the leading order polynomial behavior of both sides with respect to $n$ dominates.
		
		This concludes the proof of Theorem \ref{th:ass-spectrum-sobolev-2}.
	\end{proof}

\end{document}